\documentclass[letter,11pt]{amsart}
\usepackage{fullpage}
\usepackage[toc,page]{appendix}
\usepackage{amssymb, amsfonts} 
\usepackage{graphicx}
\usepackage{color}
\usepackage[font=footnotesize]{caption}
\usepackage{mathtools}
\usepackage{pinlabel}
\usepackage{graphicx}
\usepackage{tikz}
\usepackage{tikz-cd}
\usepackage{caption}
\usepackage{subcaption}
\usepackage{multirow}
\usepackage{longtable}
\usepackage{adjustbox}
\usepackage{comment}

\usepackage{hyperref}

\usepackage{float}
\usepackage{calc}
\def\thetitle{{ . }}
\hypersetup{
	pdftitle=  \thetitle,
	pdfauthor=  {Juhun Baik}
}

\newtheorem{thm}{Theorem}[section]
\newtheorem{lem}[thm]{Lemma}
\newtheorem{cor}[thm]{Corollary}
\newtheorem{prop}[thm]{Proposition}
\newtheorem{que}{Question}

\newtheorem{defn}[thm]{Definition}
\newtheorem*{que*}{Question}

\theoremstyle{remark}

\newtheorem*{rmk}{\textbf{Remark}}

\usepackage{tikz}
\usetikzlibrary{calc,decorations.markings}
\usetikzlibrary{shapes,snakes}

\theoremstyle{definition}
\newtheorem*{defn*}{Definition}

\newcommand\topnorm[1]{\overline{\langle\langle #1 \rangle\rangle}}

\newcommand\Homeo{\operatorname{Homeo}}

\newcommand{\Map}{\mathrm{Map}}
\newcommand{\FMap}{\mathrm{FMap}}
\newcommand{\PMap}{\mathrm{PMap}}

\title[Topological normal generators]{Topological normal generation of big mapping class groups}

\author{Juhun Baik}
\address{Juhun Baik, Department of Mathematical Sciences, KAIST,  
	291 Daehak-ro, Yuseong-gu, Daejeon 34141, South Korea }
\email{jhbaik@kaist.ac.kr}

\date{\today}

\begin{document}
	
	\begin{abstract}		
		A topological group $G$ is \emph{topologically normally generated} if there exists $g \in G$ such that the normal closure of $g$ is dense in $G$.
		Let $S$ be a tame, infinite type surface whose mapping class group $\Map(S)$ is generated by a coarsely bounded set (CB generated). 
		We prove that if the end space of $S$ is countable, then $\Map(S)$ is topologically normally generated if and only if $S$ is uniquely self-similar.
		Moreover, when the end space of $S$ is uncountable, we provide a sufficient condition under which $\Map(S)$ is topologically normally generated.
		As a consequence, we construct uncountably many examples of surfaces that are not telescoping yet have topologically normally generated mapping class groups. 
		Additionally, we establish the semidirect product structure of $\FMap(S)$, the subgroup of $\Map(S)$ that pointwisely fixes all maximal ends that each is isolated in the set of maximal ends of $S$.
		This result leads to a proof that the minimum number of topological normal generators of $\Map(S)$ is bounded both above and below by constants that depend only on the topology of $S$.
		Furthermore, we demonstrate that the upper bound grows quadratically with respect to this constant.
	\end{abstract}
	
	\maketitle
	
\section{Introduction}\label{sec:1.intro}

Let $S$ be an orientable surface.
We say that $S$ is of finite type if its fundamental group is finitely generated. 
Otherwise, we will call $S$ is called a surface of infinite type.
Equivalently, $S$ has infinite type if it has either infinitely many punctures, infinitely many genus, or both.
A mapping class group of an infinite type surface $S$ is usually called a \emph{big mapping class group}.
Throughout this article we denote the big mapping class group of $S$ by $\Map(S)$.
For a more detailed study of the big mapping class groups, we refer to the reader to \cite{aramayona2020bigmcgoverview}.

A big mapping class group is a topological group equipped with the topology induced by the compact-open topology on the group of orientation-preserving homeomorphisms of $S$. 
With this topology, a big mapping class group is Polish.
The key distinction between mapping class groups of finite-type surfaces and big mapping class groups is that the latter are neither countably generated nor locally compact.
Thanks to the influential work of Mann and Rafi \cite{mann2023large}, big mapping class groups can now be studied from the perspective of large-scale geometry through the concept of \emph{coarsely boundedness (abbreviated CB)}, introduced by Rosendal \cite{rosendal2022coarse}.

\subsection{Topological normal generation}
A topological group is \emph{topologically normally generated} if it has an element whose normal closure is dense.
If a topological group has the \emph{Rokhlin property}, meaning it contains a dense conjugacy class, it is automatically topologically normally generated.
For finite type surfaces, it is well known that almost every mapping class group is normally generated. 
(See \cite{lanier2022normal}, \cite{baik2021reducible})
By definition, if a group is normally generated, then it is topologically normally generated. 
However, the converse does not always hold.
For instance, Domat \cite{domat2022pure} proved that there is a surjection from the mapping class group of the Loch Ness monster surface to the additive group of rational numbers, yet this surface is uniquely self-similar.
Lanier and Vlamis \cite{lanier2022Rokhlin} proved that if the surface $S$ is uniquely self-similar, then $S$ has the Rokhlin property, and hence $\Map(S)$ is topologically normally generated.
We show that the converse also holds when the surface $S$ has a countable end space and $\Map(S)$ is generated by a coarsely bounded set (CB generated).

\begin{thm}{(Theorem \ref{thm:Map_gen})}\label{thm:thmalpha_countable}
	Suppose $S$ is a tame, infinite type surface with countable end space, and $\Map(S)$ is CB generated.
	Then $S$ is uniquely self-similar if and only if $\Map(S)$ is topologically normally generated.
\end{thm}

For the case that the end space of $S$ is uncountable, this has been studied by Vlamis \cite{vlamis2024telescoping}, who showed that the mapping class group of telescoping surfaces are normally generated by a strong dilatation map. 
In fact, he proved that the mapping class group of telescoping surfaces are \emph{strongly distorted}, though we will not cover this result in this paper.
We note that if $S$ is infinite type surface and telescoping, it cannot have a countable end space.
This was established by Mann and Rafi \cite{mann2023large}. 
Aside from telescoping surfaces, we provide another class of infinite type surfaces with uncountable ends whose mapping class group is topologically normally generated.
\begin{thm}{(Theorem \ref{thm:uncountable_case_ex})}\label{thm:thmalpha_uncountable}
	Suppose $S$ is a tame, infinite type surface and satisfies the following.
	\begin{enumerate}
		\item $E(S)$ is uncountable and $\Map(S)$ is CB generated.
		\item $S$ is a connected sum of $S_{p}$ and $S_{u}$, where $S_p$ is perfectly self-similar and $S_u$ is uniquely self-similar with the unique maximal end $x$, and $|Im(x)| \leq 1$. 
	\end{enumerate}
	Then $\Map(S)$ is topologically normally generated.
\end{thm}
Here, $E(S)$ is a space of ends of $S$ and $Im(x)$ is the set of immediate predecessors of the maximal end $x$.
For precise definitions, see Section \ref{sec:mann_rafi} .
Note that there are uncountably many examples satisfying the conditions in Theorem \ref{thm:uncountable_case_ex}.

\subsection{First cohomology classes $H^1(\FMap(S), \mathbb{Z})$}

To prove the topological normal generation of big mapping class group, we use the following obstruction.
\begin{lem}{(Lemma \ref{lem:vlamis_obs})}
	Let $G$ be a topological group and $H$ be a non-cyclic abelian group.
	If there exists a continuous surjective homomorphism from $G$ to $H$, then $G$ cannot be topologically normally generated.
\end{lem}
We will construct the continuous surjective homomorphism by calculating the first cohomology classes.
In \cite{aramayona2020firstcohomology}, the authors studied the first cohomology of $\PMap(S)$, the \emph{pure mapping class group}, and proved that $\PMap(S)$ is a semidirect product of $\overline{\Map_c(S)}$, the topological closure of compactly supported maps, with handle shifts.
Additionally, in \cite{hernandez2021integral}, the authors extended this study to non-orientable surfaces.

We obatin an analogous result to that in \cite{aramayona2020firstcohomology} by utilizing the generalized shifts.
As a variant of $\PMap(S)$, we investigate $\FMap(S)$, a subgroup collecting all mapping classes that fix the \emph{isolated maximal ends} pointwise.
Here, \emph{isolated} refers to maximal ends that are isolated within the set of maximal ends of $S$ but not in the entire end space $E(S)$, since every maximal end is accumulated by other ends, so cannot be isolated in the whole end space unless itself is an isolated puncture.

The maximality of an end is derived from the partial order on the end space introduced in \cite{mann2023large}. 
Mann and Rafi proved that under this partial order, the maximal element exists.
Choose an immediate predecessor $z$ which both accumulates to maximal ends $A$ and $B$. (For precise definition, we refer Section \ref{sec:mann_rafi}.)
We number each $z$ by integers in a sequence so that $\{z_n\}_{n\in \mathbb{Z}}$ accumulates to $A$ (or $B$) when $n$ goes to $\infty$ (or $-\infty$, respectively).
The generalized shift $\eta_{A, B, z}$ is a shift map which sends $z$ of index $n$ to $z$ with index $n+1$.
In simple terms, the \emph{cohomology class} we investigate \emph{counts how many $z$'s are shifted}.
Collecting all such \emph{countings} (including handle shifts), we have the following structure for $\FMap(S)$.

\begin{thm}{(Theorem \ref{thm:semi_directprod_of_FMap})}\label{thm:thmalpha_firstcohom}
	Suppose $S$ is a tame, infinite type surface.
	Let $G_0 := \{ x \in \mathcal{M}(S) ~|~ x \in E^G \text{ and } Im(x) \cap E^G(S) = \emptyset\}$.
	Then for any distinct $A, B \in \mathcal{M}(S)$, 
	\[
		\FMap(S) = \mathcal{F} \rtimes \left( \prod_{\substack{A, B \in \mathcal{M}(S) \\ \text{if } z \in E_{cp}(A, B)}} \langle \eta_{A, B, z} \rangle \oplus \prod_{\substack{A, B \in \mathcal{M}(S) \\ \text{if both } A, B \in G_0}} \langle h_{A, B} \rangle \right)
	\]
\end{thm}
We note that each $\langle \eta_{A, B, z} \rangle$ and $\langle h_{A, B} \rangle$ is $\mathbb{Z}$.
The kernel part $\mathcal{F}$ contains all \emph{finitely bounded mapping classes} (See Definition \ref{defn:finite_boundness}).

\subsection{Normal generating set of big mapping class groups}
There have been numerous studies on the (normal) generation of the big mapping class group, as finding a \emph{nice} generating set for a given group is crucial for understanding its structure and properties.
In \cite{aramayona2020firstcohomology}, the authors prove that the set of Dehn twists topologically generates $\Map(S)$ if and only if $S$ has at most one non-planar end.
Moreover, Patel and Vlamis \cite{patelvlamis2018algebraic} proved that if $S$ has at least $2$ ends accumulated by genus, then $\PMap(S)$ is topologically generated by Dehn twists and handle shifts. 
In \cite{calegari2022normal}, the authors show that for the finite type surface $S$ with a Cantor set $K$ embedded, $\Map(S-K)$ is normally generated by torsion normal generator of $\Map(S)$ unless $S$ has genus $2$ with $5k+4$ punctures for some $k\geq 0$.
In \cite{hernandez2022conjugacy}, the authors studied topological properties of conjucagy class in $\Map(S)$.
They showed that all conjugacy classes in $\Map(S)$ are meager and provided a condition for $\Map(S)$ to contain a dense conjugacy class.
Additionally, as mentioned earlier, Vlamis (\cite{vlamis2023homeomorphism}, \cite{vlamis2024telescoping}) proved that if $S$ is perfectly self-similar or telescoping, then $\Map(S)$ has a normal generator.
Also, Mann and Rafi \cite{mann2023large} give a necessary and sufficient condition for a tame surface $S$ to admit a coarsely bounded generating set of $\Map(S)$ and explicitly constructed the CB generating set.

When considering the topological normal generating set, we pose the same question: Which set of mapping classes topologically normally generates the whole mapping class group? How large should such set be?
By their nature, big mapping class groups are neither countably generated nor compactly generated.
However, for a tame infinite type surface $S$ with CB generated $\Map(S)$, we first show that the number of normal generators is bounded by a constant that depends only on the topology of $S$.

\begin{thm}{(Theorem \ref{thm:n_of_generator})}\label{thm:thmalpha_normgenset}
	Suppose $S$ is a tame, infinite type surface with CB generated $\Map(S)$.
	Let $M$ be the number of distinct types of maximal ends in $S$, \emph{i.e.,}
	\[
	M = \big\vert \mathcal{M}(S)/\sim \big\vert
	\]
	Also, let $C := \max_{A, B \in \mathcal{M}(S)} |E_{cp}(A, B)|$.
	Then $\Map(S)$ is topologically normally generated by $M(M+C-1)$ number of generators.
\end{thm}
For the lower bound, we have the following.
\begin{thm}{(Corollary \ref{cor:lower_bound_of_normal_gens})}\label{thm:thmalpha_lowerbound}
	$\Map(S)$ is normally generated by at least $\max(1, M_{iso}-1)$ elements.
\end{thm}
Here, $M_{iso}$ is the number of distinct types of \emph{isolated maximal ends} in $S$ that is not the same type with a puncture.

If we denote $\mathbf{n}(S)$ as the minimum number of normal generator set of $\Map(S)$, then by combining Theorem \ref{thm:thmalpha_normgenset} and \ref{thm:thmalpha_lowerbound}, we have the following inequalities.
\[
	\max(1,M_{iso} - 1) \leq \mathbf{n}(S) \leq M(M+C-1)
\]
We remark that $\frac{\mathbf{n}(S)}{(M+C)^2}$ is bounded. 
\emph{i.e.,} the minimal number of topological normal generators grows quadratically by the sum of $2$ constants that only depend on the topology of the surface.

\subsection{Bounds of the rank of the abelianization of $\Map(S)$}
We introduce an application of the topological normal generation.
We first recall the result by Field, Patel and Rasmussen as follows. 
\begin{thm}{(Theorem 1.1, 1.4 in \cite{field2022sclbigmcg})}
	Let $S$ be an infinite type surface with tame end space such that $\Map(S)$ is CB generated.
	Suppose $S$ is an infinite type surface that every equivalence class of maximal ends of $S$ is infinite, except (possibly) for finitely many isolated planar ends.
	Then,
	\begin{enumerate}
		\item The commutator subgroup $[\Map(S), \Map(S)]$ is Polish and both open and closed subgroup of $\Map(S)$.
		\item The abelianization of $\Map(S)$ is finitely generated and is discrete when endowed with the quotient topology.
	\end{enumerate}
\end{thm}
Applying their result of \cite{field2022sclbigmcg}, Corollary \ref{thm:thmalpha_normgenset} gives an upper bound of the rank.
\begin{cor}{Corollary \ref{cor:Z2rank_of_ab}}\label{thm:thmalpha_Z2ab}
	Suppose $S$ is a finite connected sum of perfectly self-similar surfaces. 
	Then the abelianization of $\Map(S)$ is finitely generated by at most $M(M+C-1)$.
\end{cor}
The proof is straightforward, since the number of topological normal generators can not exceed the rank of its abelianization.
However, we cannot get an effective lower bound for the rank by using the inequality in Theorem \ref{thm:thmalpha_lowerbound}, as $M_{iso}$ counts the isolated maximal ends, which is not of the same type as a puncture.

\subsection*{Outline}
In Section \ref{sec:mann_rafi}, we recall some definitions and introduce the work of Mann and Rafi \cite{mann2023large}.
The obstruction to topological normal generation is presented in Section \ref{sec:obstruction}.
Before proving the main theorem, we breifly introduce a generalized shift map in Section \ref{sec:shifts}.
In Section \ref{sec:cohomology} we construct a first integral cohomology for $\FMap(S)$ and prove Theorem \ref{thm:thmalpha_firstcohom}.
Using cohomology classes, we prove Theorem \ref{thm:thmalpha_countable} in Section \ref{sec:parity}.
In Section \ref{sec:uncountable}, we provide additional examples of surfaces with uncountable end space whose mapping class groups are topologically normally generated, by proving Theorem \ref{thm:thmalpha_uncountable}.
In Section \ref{sec:n_of_norm_gen}, we study the set of topological normal generators for CB generated cases and prove Theorem \ref{thm:thmalpha_normgenset}, \ref{thm:thmalpha_lowerbound} and Corollary \ref{thm:thmalpha_Z2ab}.
We conclude the article with a remark on the cases of locally CB (but not CB generated) in Section \ref{sec:locCB_not_CBgend}. 

\subsection*{Funding}
This work was supported by National Research Foundation of Korea(NRF) grant funded by the Korea government(MSIT) [No. 2020R1C1C1A01006912].

\subsection*{Acknowledgement}
The author would like to thank the University of Utah, and especially Mladen Bestvina, for their hospitality during the visit.
Also, the author appreciates Harry Hyungryul Baik, Mladen Bestvina, Thomas Hill, Priyam Patel and Nicholas Vlamis for their helpful comments. 
In particular, without the kind support and influential comments from Sanghoon Kwak, this project would not have benefited from the improvements that we have made.

\section{Mann-Rafi's work}\label{sec:mann_rafi}
\subsection{Infinite type surfaces}
Throughout the paper, we denote $S$ as an infinite type surface without boundary unless otherwise specified.
The \emph{end} of $S$ is an equivalence class of a nested sequence of open, connected sets $\{U_i\}_{i\in \mathbb{Z}_{\geq 0}}$ such that $S \supset U_0 \supset U_1 \supset \cdots \supset U_k \supset \cdots$ with $\overline{U_{i+1}} \subset U_i$ for all $i$ and $\cap_{i = 0} U_i$ is empty.
Two such sequences $\{V_i\}$ and $\{W_j\}$ are equivalent if for any $n$, there exists $m$ such that $W_m \subset V_n$ and for any $m$, there exists $n$ such that $V_n \subset W_m$.

Let $E = E(S)$ denote the collection of all ends of $S$, and we call it \emph{end space of $S$}.
$E^G = E^G(S)$ is a subset of the end space, which collects all the ends that are accumulated by genus.
That is, $\{U_i\}$ belongs to $E^G(S)$ if every $U_i$ has genus, and hence infinitely many genus.
For a subsurface $\Sigma \subset S$ and $x = \{U_i\} \in E(S)$, we say $x \in \Sigma$ if there exists $n$ such that for all $k\geq n$, $U_k \subset \Sigma$.
According to \cite{kerekjarto1923vorlesungen}, \cite{richards1963classification}, the pair $(E, E^G)$ forms a closed nested set of the Cantor set. 
Also, every infinite type surface with compact boundary is characterized by the $5$-uple $(p, b, g, E, E^G)$, where $p$ is the number of punctures, $b$ is the number of boundaries, and $g$ is the genus.
We endow a compact-open topology on $\mathrm{Homeo}^+(S)$, the orientation preserving homeomorphism group of $S$.
The big mapping class group $\Map(S)$ of $S$ is a group $\mathrm{Homeo}^+(S) / \mathrm{Homeo}_0(S)$, where $\mathrm{Homeo}_0(S)$ is a path component of the identity element of $\mathrm{Homeo}^+(S)$.
Therefore $\Map(S)$ is a topological group equipped with the quotient topology and is Polish.

By its nature, the big mapping class group is neither locally compact nor compactly generated.
Due to its complex structure, standard geometric group theoretic tools cannot be directly applied to the big mapping class group.
To address this issue, we follow the notion proposed by Rosendal \cite{rosendal2022coarse} which is analogous to the compactness, so called \emph{coarsely bounded}.
\begin{defn}\label{defn:CB}
	Let $G$ be a topological group and $A \subset G$ be a subset.
	$A$ is called \emph{coarsely bounded}, (abbreviated CB), if every compatible left-invariant metric on $G$ gives $A$ a finite diameter.
\end{defn}
A group $G$ is said to be \emph{CB generated} if it contains a CB set that generates $G$, and $G$ is \emph{locally CB} if there exists a CB neighborhood of an identity.

\subsection{Mann-Rafi's classification}
In \cite{mann2023large}, the authors provide a necessary and sufficient condition for $S$ such that $\Map(S)$ is locally CB.
We first introduce a preorder on $E$, defined as follows.
\begin{defn}[Preorder on $E(S)$, \cite{mann2023large}]
	Let $\preccurlyeq$ be the binary relation on $E$ such that $y \preccurlyeq x$ if and only for any neighborhood $U$ of $x$, there exists a neighborhood $V$ of $x$ and $f \in \Map(S)$ so that $f(V) \subset U$.
\end{defn}
In other words, $y\preccurlyeq x$ implies that a copy of neighborhood of $y$ (and hence $y$ itself) must appear nearby $x$.
It follows from the definition that $\preccurlyeq$ is a preorder on the end space.
We follow the notation in the paper that for $x\sim y$ if and only if $x \preccurlyeq y$ and $y \preccurlyeq x$.
By definition, $\sim$ is an equivalence relation.
Let $E(x)$ denote a subset $\{y \in E(S)~|~ y\sim x\}$.
They also prove that there are only finitely many types of maximal ends if $\Map(S)$ is locally CB.

\begin{prop}[Proposition 4.7, in \cite{mann2023large}]\label{prop:mannrafi_prop4.7}
	Suppose $\Map(S)$ is locally CB.
	$E(S)$ has maximal elements under $\preccurlyeq$.
	Furthermore $E(x)$ is either finite or a Cantor set for every maximal element $x$.
\end{prop}
We denote that $\mathcal{M}(S)$ be a set of maximal elements in $E= E(S)$.
Let $x \in \mathcal{M}(S)$.
We call $z$ an \emph{immediate predecessor} of $x$ if $z \preccurlyeq w \preccurlyeq x$ and $w\not\sim x$ implies $z \sim w$.
For a maximal end $x$, we denote $Im(x)$ as the set of types of the immediate predecessors of $x$.
We slightly abuse the notation of immediate predecessors so that $Im(x)$ contains a \emph{handle} if $x$ is accumulated by genus and no immediate predecessors are accumulated by genus.
If $\Map(S)$ is locally CB, there are only finitely many different types of maximal ends.
\begin{lem}[Lemma 5.3 in \cite{mann2023large}]
	If $\Map(S)$ is locally CB, then the number of distinct types in $\mathcal{M}(S)$ is finite.
\end{lem}

The surface $S$ is called \emph{self-similar} if for any finite partition of $E(S)$, at least one of the partitions contains a set homeomorphic to $E(S)$.
For example, the infinite flute surface has $E(S) \cong \omega+1$ and $E^G(S)$ is empty.
In any partition of $E(S)$, at least one of the sets must contain infinitely many ends and the maximal end, thus making it self-similar.
There are $2$ types of self-similar surfaces.
\begin{defn}\label{defn:self-similar}
	Let $S$ be a self-similar surface.
	Then $\mathcal{M}(S)$ is either a singleton or a Cantor set.
	We call $S$ \emph{uniquely self-similar} if $\mathcal{M}(S)$ is singleton, otherwise we call $S$ \emph{perfectly self-similar}.
\end{defn}

One of the main theorems in \cite{mann2023large} is that they provide a necessary and sufficient condition for $\Map(S)$ to be locally CB.
\begin{thm}[\cite{mann2023large}]\label{thm:Mann-Rafi_locCB}
	$\Map(S)$ is locally CB if and only if there exists a finite type surface $K \subset S$ with the following properties,
	\begin{enumerate}
		\item Each complementary region of $K$ has one or infinitely many ends and infinite or zero genus.
		\item The complementary regions of $K$ partition $E$ into clopen sets, indexed by finite sets $\mathcal{A}$ and $\mathcal{P}$ such that 
		\begin{itemize}
			\item each $A\in \mathcal{A}$ is self-similar, with $\mathcal{M}(A) \subset \mathcal{M}(E)$ and $\mathcal{M}(E) \subset \bigcup_{A\in \mathcal{A}}\mathcal{M}(A)$,
			\item each $P \in \mathcal{P}$ is homeomorphic to a clopen subset of some $A\in\mathcal{A}$,
			\item (small zoom condition) for any $x_A \in \mathcal{M}(A)$ and any neighborhood $V$ of the end $x_A$ in $S$, there is a homeomorphism $f_V$ of $S$ such that $f_V(V)$ contains the complementary region to $K$ with end set $A$.
		\end{itemize}
	\end{enumerate}
\end{thm}

We will first introduce some key concepts.
\begin{defn}\label{defn:limit,stable}
	Let $S$ be an infinite type surface and $E$ be its end space.
	\begin{enumerate}
		\item $E$ is said to be \emph{limit type} if there exists a finite index subgroup $G$ of $\Map(S)$, a $G$-invariant set $X \subset E$, points $\{z_n\}_{n \geq 1} \subset E$ which are pairwise inequivalent, and a family of neighborhoods $U_n$ of $E$ such that 
		\begin{itemize}
			\item $U_{i+1} \subset U_i$ and $\cap_{i} U_i = X$,
			\item $E(z_n) \cap U_n \neq \emptyset$, $E(z_n) \cap U_1^c \neq \emptyset$, and $E(z_n) \subset (U_n \cup U_0^c)$.
		\end{itemize}
		\item $x \in E$ is \emph{stable} if for any smaller neighborhood $U' \subset U$ of $x$, there exists a homeomorphic copy of $U$ contained in $U'$.
	\end{enumerate}
\end{defn}
Roughly speaking, if there are infinitely many distinct types of ends that accumulate to a single end, then $E$ is a limit type. 
If $\Map(S)$ is locally CB and there are no limit types in $E$, then we can proceed to analyze which types of ends accumulate to the maximal end.
\begin{lem}[Lemma 6.10 in \cite{mann2023large}]
	Let $\mathcal{A}$ be the partition obtained by Theorem \ref{thm:Mann-Rafi_locCB}.
	Suppose $\Map(S)$ is locally CB and $E(S)$ is not of limit type.
	\begin{enumerate}
		\item For every pair $A, B \in \mathcal{A}$, there is a clopen set $W_{A,B}$ satisfying that $E(z) \cap W_{A, B} \neq \emptyset$ if and only if
		\[
		E(z) \cap (A - \{x_A\}) \neq \emptyset \text{\quad and \quad} E(z) \cap (B - \{x_B\}) \neq \emptyset
		\]
		\item For every $A \in \mathcal{A}$, there is a clopen set $W_A$ satisfying that 
		\[
		E(z) \cap (A - \{x_A\}) \neq \emptyset \text{ and } ^\forall B\neq A, E(z) \cap (B - \{x_B\}) = \emptyset \quad \Rightarrow \quad E(z) \cap W_A \neq \emptyset
		\]
	\end{enumerate}
\end{lem}

\begin{defn}[$E_{cp}(A, B)$ in \cite{mann2023large}]\label{defn:W_AB}
	The \emph{countable predecessor set} $E_{cp}(A, B)$ is a subset of $W_{A, B}$ consisting of points $z$ where $z$ is maximal in $W_{A, B}$ and $E(z)$ is countable in $W_{A,B}$.
\end{defn}
In other words, $E_{cp}(A, B)$ is the set of common immediate predecessors that accumulate to both $A$ and $B$.
Note that $E_{cp}(A, B)$  contains only finitely many types and is a subset of both $Im(x_A)$ and $Im(x_B)$.
In \cite{mann2023large}, the authors proved that if $\Map(S)$ is CB generated then (1) for any $A, B\in\mathcal{A}$, $E_{cp}(A,B)$ contains only finitely many distinct types, and (2) $W_{A, B} \cong W_{B, A}$.
We also note that the locally CB criterion implies that each partition $A\in\mathcal{A}$ is a self-similar subsurface and if $S$ has countably many ends, then each $A$ is uniquely self-similar.

\begin{defn}\label{defn:tame}
	An end space $E$ is \emph{tame} if, for every $A$ in $\mathcal{A}$, the maximal end $x_A$ of $A$ has a stable neighborhood, and for any $A, B \in \mathcal{A}$, every maximal point in $W_{A,B}$ has a stable neighborhood.
\end{defn}
Mann and Rafi proved that if $S$ is tame and $\Map(S)$ is locally CB, then $\Map(S)$ is CB generated if and only if $E$ is \emph{finite rank} and not of \emph{limit type}.
We will not discuss the finite rank in detail here.
Simply put, if $E = \omega^k + 1$, then $S$ is of finite rank if and only if $k$ is a successor ordinal.

If $S$ is tame and has CB generated $\Map(S)$, we can decompose the surface more simply than in Theorem \ref{thm:Mann-Rafi_locCB}.
\begin{prop}[Proposition 5.4 in \cite{mann2023large}]\label{prop:prop5.4_mannrafi}
	If $\Map(S)$ is locally CB, then there exists a partition $\mathcal{A}$ and $L$ of $S$ satisfying that
	\begin{enumerate}
		\item $|\mathcal{A}| < \infty$.
		\item $E(S) = \bigcup_{A \in \mathcal{A}}E(A)$.
		\item Each $A \in \mathcal{A}$ is self-similar.
		\item $L = S - \bigcup_{A\in\mathcal{A}} A$ is a sphere with $|\mathcal{A}|$ boundaries.
	\end{enumerate}
	Moreover, if we further assume that $S$ is tame and $\Map(S)$ is CB generated, then the CB generating set of $\Map(S)$ is the union of the following.
	\begin{itemize}
		\item The identity neighborhood $\mathcal{V}_L := \{ f~|~ f\vert_L  = id \}$,
		\item Generating sets of $\Map(L)$,
		\item Generalized shifts $\eta_{A, B}$ and $\eta_{A, B, z}$ for $A, B \in \mathcal{A}$ and $z \in E_{cp}(A, B)$,
		\item Handle shifts,
		\item Half twists $g_{A,B}$ when $x_A \sim x_B$, where $x_A, x_B$ is a unique type of maximal ends of $A$ and $B$, respectively.
	\end{itemize}
\end{prop}
We will use this fact in Section \ref{sec:n_of_norm_gen} to estimate the bounds of the minimum number of topological normal generators.

\section{Obstruction}\label{sec:obstruction}
We introduce the obstruction of being (topologically) normally generated group.
\begin{lem}\label{lem:vlamis_obs}
	Let $G$ be a topological group and $H$ be a non-cyclic abelian group.
	If there exists a continuous surjective homomorphism from $G$ to $H$, then $G$ cannot be topologically normally generated.
\end{lem}
\begin{proof}
	If $G$ is topologically normally generated, then the image under homomorphism to the abelian group must be cyclic.
\end{proof}

For example, In \cite{aramayona2020firstcohomology} the $\PMap(S)$ is a semi-direct product of $\overline{\Map_c(S)}$ with handle shifts.
Therefore if the surface has at least $3$ ends accumulated by genus then $\PMap(S)$ is not topologically normally generated. 
From now on, the main strategy to show that the mapping class group is not (topologically) normally generated is to construct the surjective map to the noncyclic abelian group.

Through this obstruction, we can identify infinite type surfaces whose mapping class groups are not topologically normally generated.
We first note that there is a short exact sequence
\[
	1 \to \PMap(S) \xrightarrow{\iota} \Map(S) \xrightarrow{\pi} \Homeo(E(S), E^G(S)) \to 1
\]
where both $\iota, \pi$ are continuous maps.
In fact, for any $\Map(S)$-invariant set $F \subset E(S)$ the forgetful map $\Map(S) \to \Homeo(F)$ is continuous.

\begin{prop}\label{thm:main_obstruction}
	Suppose $S$ is an infinite type surface with countable $E(S)$. 
	Let $x_1,\cdots ,x_k$ be the distinct types of maximal ends of $S$ and $n_i = |E(x_i)|$. 
	If $\Map(S)$ is topologically normally generated, then $n_1 \geq n_2 = \cdots =n_k =1$ after reindexing.
\end{prop}
\begin{proof}
	Since the maximal ends must be setwise preserved, we consider only how the mapping classes permute these ends, disregarding other aspects of the maps.
	This leads to the following surjective homomorphism.
	\[
		\Map(S) \twoheadrightarrow S_{n_1}\times \cdots \times S_{n_k}
	\]
	After abelianizing the image, we obtain a product of $\mathbb{Z}/2\mathbb{Z}$, which only occurs for $n_i > 1$.
	Hence the only one maximal end type could have more than $1$ distinct maximal ends of that type.
\end{proof}

\begin{rmk}
We note that Theorem \ref{thm:main_obstruction} is a negative criterion, and its converse does not hold in general.
In Figure \ref{fig:Mann-rafi_ex}, there are $3$ distinct maximal ends, which we label as follows. 
$A$ (the rightmost maximal end accumulated by punctureds only),
$B$ (the leftmost end accumulated by both punctures and genus), and $C$ the bottom maximal end accumulated by genus only).
There are two distinct shifts in the surface, one is the puncture shift $\eta_{A, B, p}$ (here, $p$ denotes the puncture) and the other is a handle shift $h$ between $B, C$. 
Let $c$ be the boundary of a small neighborhood of $B$.
Suppose $f \in \Map(S)$.
We choose a seperating closed curve $\gamma$, which separates the maximal end $B$ from the others so that both $\gamma$ and $f(\gamma)$ do not intersect with $c$. 
Then $\gamma$ and $c$ (and $f(\gamma)$ and $c$ also) cobound a finite type subsurface $\Sigma_\gamma$ ($\Sigma_{f(\gamma)}$, respectively).  
We can construct the homomorphism from $\Map(S)$ to $\mathbb{Z}\times \mathbb{Z}$ defined as follows.
\[
	f \mapsto ( \text{ \# punctures in } \Sigma_\gamma -  \text{\# punctures in } \Sigma_{f(\gamma)}, \text{\# genus in } \Sigma_\gamma -  \text{\# genus in } \Sigma_{f(\gamma)})
\]
Since $\eta_{A, B, p}^n \circ h^m$ has a value of $(n,m)$, the map is surjective.
By the criterion, the surface in Figure \ref{fig:Mann-rafi_ex} has a mapping class group that is not topologically normally generated.
\end{rmk}
\begin{figure}[h]
	\centering
	\includegraphics[width=.5\textwidth]{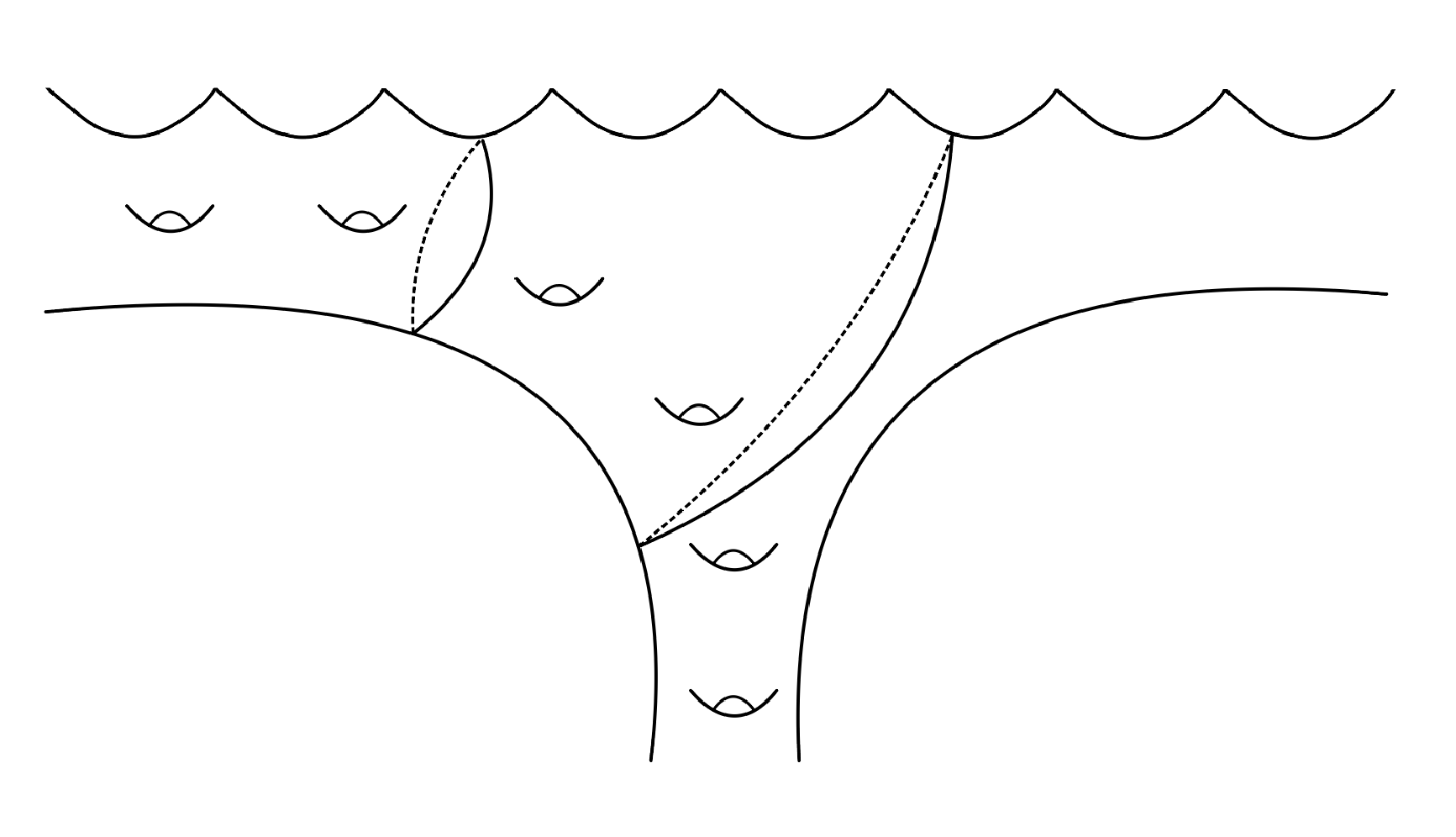}
	\caption{Each maximal end is distinct from the others. A separating curve on the left is mapped to the right one by handle shift twice and puncture shift three times, both shifts toward to the right. If we assign the rightward direction as positive, this mapping corresponds to the value $(3,2)$.}
	\label{fig:Mann-rafi_ex}
\end{figure}

\section{Various shifts}\label{sec:shifts}
In this section we introduce the various ways in which the shift maps on immediate predecessors can arise.
We will show how much they differ from shifts.
From now on, the surface $S$ we are dealing with is a tame, infinite type surface with \emph{countable end space} with CB generated $\Map(S)$.
We begin with the definition of biinfinite strips.
\begin{defn}[Biinfinite strip]\label{defn:biinfinite_strip}
	The biinfinite strip $\mathcal{B}$ is an infinite type surface homeomorphic to the surface defined as follows.
	\[
	\mathcal{B} \cong \mathbb{R} \times [-1, 1] - \left(\bigcup_{m\in\mathbb{Z}}D_{\frac{1}{2}}(m, 0)\right)
	\]
	where $D_{r}(p, q)$ is an open disk of radius $r$ centered at $(p,q)$.
\end{defn}
We label $\partial_m$ to each boundary of the deleted disk centered at $(m, 0)$. 
We also label $\partial_{\pm}$ as long boundaries, $\mathbb{R}\times\{\pm1\}$.

\begin{figure}[h]
	\centering
	\includegraphics[width=.5\textwidth]{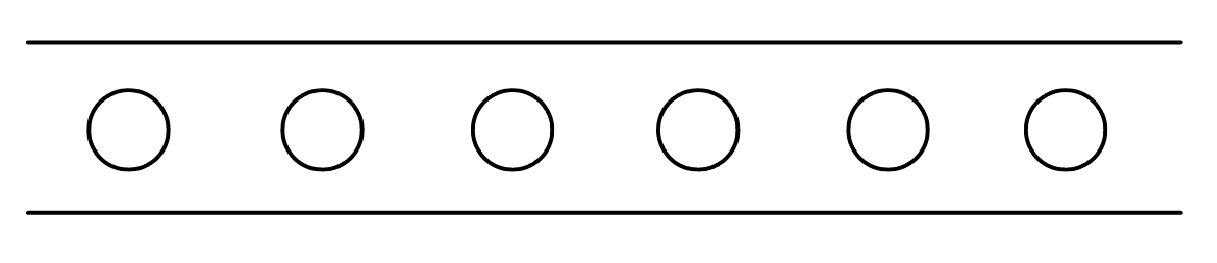}
	\caption{Biinfinite strip $\mathcal{B}$}
	\label{fig:biinfinite_strip}
\end{figure}
\begin{figure}[h]
	\centering
	\includegraphics[width=.8\textwidth]{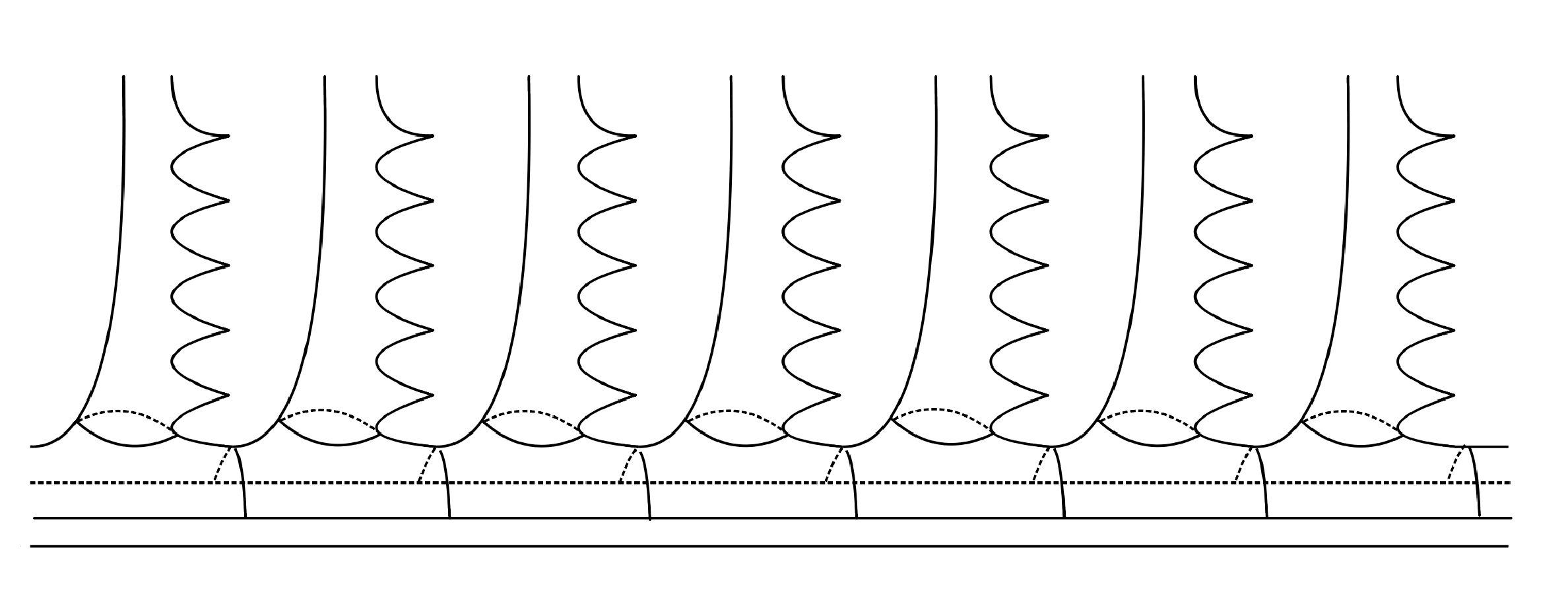}
	\caption{Embedded biinfinite strip $\mathcal{B}_z(A, B)$, where $z \sim \omega+1$ and both $A, B$ are $\omega^2+1$, and $(E(S), E^G(S)) = (\omega^2\cdot 2 + 1, \emptyset)$.}
	\label{fig:biinfinite_strip_embedded}
\end{figure}

There are many ways to embed $\mathcal{B}$ into $S$. 
We choose one way for given $\eta_{A, B, z}$ with $z \in E_{cp}(A,B)$ as follows.
\begin{defn}[biinfinite strip $\mathcal{B}_z(A,B)$]\label{defn:biinfinite_strip_for_z}
	Let $A,B$ be $2$ distinct maximal ends of $S$ and suppose $E_{cp}(A, B)$ is nonempty.
	Fix $z\in E_{cp}(A,B)$.
	Then there is an orientation preserving embedding $\iota$ of a biinfinite strip $\mathcal{B}$ into $S$ which satisfies the following.
	\begin{enumerate}
		\item Let $\Sigma_m$ be the subsurface cut by $\partial_m$ which does not contain $\iota(\mathcal{B})$. Then the only maximal end of $\Sigma_m$ is of type $E(z)$ for all $m \in \mathbb{Z}$.
		\item $\iota(t, \pm1)$ approaches to $A$ as $t\to -\infty$ and $\iota(t, \pm1)$ approaches to $B$ as $t \to \infty$.
		\item The support of $\eta_{A, B, z}$ contains the image of $\mathcal{B}$ and $\eta_{A, B, z}$ sends the maximal end in $\Sigma_m$ to the maximal end in $\Sigma_{m+1}$.
	\end{enumerate}	
	We denote the image of $\mathcal{B}$ union with $\Sigma_m$'s as $\mathcal{B}_z(A,B)$.
	We also label the unique maximal end of a subsurface $\Sigma_m$ to $z_m$.
\end{defn}
We sometimes abuse the notation $\mathcal{B}_z(A,B)$ by writing it simply as $\mathcal{B}_z$ if there is no confusion.
Note that $\mathcal{B}_z(A,B)$ may not be unique.
For example, if $z$ is a type of the unique maximal end of the infinite flute surface (whose end space is $\omega+1$), then there are uncountably many different embeddings of $\mathcal{B}$ in $S$ by selecting which punctures are excluded from the image of $\mathcal{B}$.
Nevertheless, all such embeddings share the following property.

\begin{lem}\label{lem:gamma_intersects_B}
	Let $\gamma$ be an oriented simple closed curve on $S$ which separates $A$ and $B$.
	Then $\gamma$ essentially intersects $\mathcal{B}_z(A,B)$ for any $z \in E_{cp}(A, B)$. 
	In particular, $\gamma$ separates $\mathcal{B}_z(A,B)$.
\end{lem}
\begin{proof}
	Suppose not.
	Then $\mathcal{B}_z(A, B)$ is embedded in $S - \gamma$, which implies $A, B$ are in the same partition produced by $\gamma$.
	Obviously, it does not depend on the choice of $z \in E_{cp}(A, B)$.
\end{proof}

Let $S$ be an infinite type surface and suppose $S$ has at least $2$ distinct maximal ends $A, B$ which share an immediate predecessor of type $z$.
There is a $z$-shift map $\eta_{A, B, z}$ from $A$ to $B$, and $\eta_{A, B, z}$ gives a $\mathbb{Z}$-indexing on $z$-type ends between $A, B$.
In particular, $\eta_{A, B, z}(z_k) = z_{k+1}$ if we restrict $\eta_{A, B, z}$ to $E(S)$.

We index each $z$ so that $z_k$ accumulates to one maximal end $A$ as $k$ goes to $-\infty$,
The other maximal end $B$ is accumulated by $z_k$'s with positive indices.

Choose any strictly monotone bi-infinite sequence $\{a_i\}_{i \in \mathbb{Z}}$.
We can consider the shift map that shifts every $z_{a_i}$ to $z_{a_{i+1}}$, otherwise fixed.

\begin{defn}[$3$ shift maps]\label{defn:shift_maps}
	Let $\{a_i\}_{i \in \mathbb{Z}}$ be a strictly monotone sequence in $\mathbb{Z}$ and let $\eta$ be a shift map defined by $\{a_i\}$.
	\begin{itemize}
		\item $\eta$ is \emph{full} if $a_i = i$.
		\item $\eta$ is \emph{permissible}\footnote{The word \emph{permissible} came from \cite{abbott2021infinite-type}.} if there are nonzero integers $n_1 > n_2$ so that the set $\{a_i\}$ is a proper subset of $\mathbb{Z}$ and contains $(-\infty, n_2] \cup [n_1, \infty)$.
		\item $\eta$ is \emph{spontaneous} otherwise.
	\end{itemize}
\end{defn}

\begin{prop}\label{prop:not_full}
	Suppose $\eta$ is not full.
	Then there is a surface homeomorphism $T$ such that $T \circ \eta$ is full.
\end{prop}
\begin{proof}
	Suppose $\eta$ is permissible.
	Then there are finitely many integers in $\mathbb{Z} - \{a_i\}_{i \in \mathbb{Z}}$.
	Let $b_1 < \cdots < b_k$ be such integers.
	Let $T_i$ be a half twist which switches $z_i$ and $z_{i+1}$.
	Then $T_{b_1}\circ \cdots \circ T_{b_k} \circ \eta$ is full. 
	
	For the spontaneous case, $a_i$ is not bounded by definition.
	Thus there are countably many disjoint intervals $I_k$, all of which are bounded and $\{a_i\} = \mathbb{Z} - \cup_{k \in \mathbb{Z}} I_k$.
	Let $T_{I_k}$ be the product of half twists $T_{i} \circ T_{i+1} \circ \cdots \circ T_{i+n}$ where $\{i, \cdots, i+n\} = \mathbb{Z}\cap I_k$.
	Since the $I_k$'s are mutually disjoint to each other, $T_{I_j}, T_{I_k}$ commutes.
	Let $T := \prod_{k \in \mathbb{Z}} T_{I_k}$. 
	Then $T\circ \eta$ is full. 
\end{proof}
Note that the map $T$ is not a shift map by construction, but it permutes the end space, so it is not in $\PMap(S)$.
One can use a $1/n$-Dehn twist on each $I_k$ instead of composition of half twists as in \cite{grant2021asymptotic}, which yields the same result.

We now show that any types of shift is still in the topological normal closure of the full shift.
We first define the finite boundedness of the map.
\begin{defn}[finite boundness]\label{defn:finite_boundness}
	$f \in \Map(S)$ is \emph{finitely bounded} with respect to $z \in E_{cp}(x, y)$ if there is a representative $f'$ of $f$ that the pullback of support of $f'$ into $\mathcal{B}$ is compact.
\end{defn}
The $T$ for permissible shift map is a finite composition of the elements that is finitely bounded but the $T$ for spontaneous case needs infinitely many half twists.
Still, it is a limit of finitely bounded maps.
\begin{prop}\label{prop:swindle}
	Suppose $\eta$ is full.
	Any finitely bounded map is in $\topnorm{\eta}$.
	Hence, any shift map $\eta'$ is in $\topnorm{\eta}$.
\end{prop}
\begin{proof}
	Suppose $f$ is finitely bounded.
	Then there is $k > 0$ that the pullback of the support of $f$ into $\mathcal{B}$ is contained in the set $\{(x, y)\in \mathbb{R}^2 ~|~ |x|\leq k \text{ and } |y| \leq 1\}$.
	Consider \footnote{In the proof of Proposition \ref{prop:swindle}, the homeomorphism acts on the right, so that the infinite product in the definition of $\widetilde{f} = f (\eta f \eta^{-1}) (\eta^2 f \eta^{-2}) \cdots$, the leftmost term $f$ is first performed on the surface and $\eta$ and so on.}
	\[
		\widetilde{f} := \prod_{i = 0}^\infty \eta^{2ki} f \eta^{-2ki} 
	\]
	This is well defined, since each support of $\eta^{2ki} f\eta^{-2ki}$ is disjoint from one another.
	Then $\widetilde{f} \eta^{2k} \widetilde{f}^{-1} \eta^{-2k} = f$ and hence $f \in \topnorm{\eta}$.
	
	By Proposition \ref{prop:not_full}, if $\eta'$ is permissible, then there is a finitely bounded map $T$ so that $\eta' = T \circ \eta$.
	Otherwise if $\eta'$ is spontaneous, then there is a limit of finitely bounded maps $T$. 
	Therefore $\eta'$ is in $\topnorm{\eta}$.
\end{proof}
We remark that permissible shifts are in fact in the normal closure of the full shift, since it does not need a limit of finitely bounded maps.

\section{Cohomology classes of $\FMap(S)$}\label{sec:cohomology}
In this section we construct the cohomology of $\FMap(S)$. 
We first recall the definition of $\FMap(S)$ from \cite{grant2021asymptotic}.
\begin{defn}\label{defn:FMap}
	Suppose $S$ is an infinite type surface and $E$ is its end space.
	Let $F := \{ x \in \mathcal{M(S)}~|~ E(x) \text{ is finite }\}$.
	$\FMap(S)$ is the kernel of $\Map(S) \to \Homeo(F)$.
\end{defn}
\emph{i.e.,} $f \in \FMap(S)$ implies that $f$ fixes all maximal ends of $S$ which are isolated in the set of maximal ends. 
By Proposition \ref{prop:mannrafi_prop4.7}, an end $x$ is maximal and isolated in the set of maximal ends if and only if $E(x)$ is finite.
Consider the set of the maximal ends and its immediate predecessors of $\mathcal{B}_z(A, B)$, which is $\{z_i\}_{i \in \mathbb{Z}} \cup \{A, B\}$.  
Now choose a separating curve $\gamma$ which induces a partition of the set into $2$ pieces, $\{z_i\}_{i < 0} \cup \{x_A\}$ and $\{z_i\}_{i \geq 0} \cup \{x_B\}$.
We give an orientation on $\gamma$ such that $A$ is on the left of $\gamma$ with respect to the orientation.
Also, for any $\gamma'$ which separates $x_A$ and $x_B$, we denote $\gamma'_L$ as the partition which contains $x_A$ (the left side of $\gamma'$) and $\gamma'_R$ as the other (the right side of $\gamma'$).

\begin{lem}\label{lem:pass_z_over_gamma}
	Let $f \in \FMap(S)$. 
	Then both $|f(\gamma)_L \cap \gamma_R|$ and $|f(\gamma)_R \cap \gamma_L|$ are finite.
\end{lem}
\begin{proof}
	Suppose $|f(\gamma)_L \cap \gamma_R|$ is infinite.
	It implies that $f(\gamma)_L$ contains infinitely many ends of type $z$ in $\gamma_R$.
	Since $f$ fixes all maximal ends, $f(\gamma)_L$ cannot contain $x_B$.
	We denote $|f(\gamma)_L \cap \gamma_R| = \{z_{m_i}\}$ with increasing sequence $m_i$.
	Then $m_i$ must diverge to infinity. 
	By definition of the partition, it implies that $f(\gamma)$ cannot be bounded in a compact subsurface, which leads to a contradiction.
\end{proof}

We define $\Phi := \Phi_z$ by $\Phi(f) := |f(\gamma)_L \cap \gamma_R| - |f(\gamma)_R \cap \gamma_L|$.
By Proposition \ref{lem:pass_z_over_gamma}, it is well-defined.

Remark that if we add these two values instead of subtracting them, then this is exactly same with the length function in \cite{grant2021asymptotic}.

\begin{prop}[Properties of $\Phi$]\label{prop:properties_of_Phi}
	Suppose $\gamma$ is an end separating curve which separates the maximal ends $A$ and $B$ and choose $z \in E_{cp}(A, B)$.
	Let $\Phi$ defined with $z$ and $\gamma$ as above.
	\begin{enumerate}
		\item $\Phi(\eta_{A, B, z}) = 1$.
		\item $\Phi : \FMap(S) \to \mathbb{Z}$ is a group homomorphism.
		\item $\Phi(f)$ is trivial if there are neighborhoods of $x_A$ and $x_B$ so that $f$ is an identity on both neighborhoods.
	\end{enumerate}
\end{prop}
\begin{proof}
	\begin{enumerate}
		\item By definition of $\eta_{A, B, z}$ and $\mathcal{B}_z(A, B)$, $\eta_{A, B, z}(\gamma)_L = \{z_i\}_{i \leq 0} \cup \{x_A\}$ and  $\eta_{A, B, z}(\gamma)_R = \{z_i\}_{i > 0} \cup \{x_B\}$.
		Hence, $\Phi(\eta_{A, B, z}) = 1$.
		\item Let $f, g \in \FMap(S)$.
		By definition, we have
		\[
		|(f\circ g)(\gamma)_L\cap \gamma_R| = |(f\circ g)(\gamma)_L\cap \gamma_R \cap f(\gamma)_L| + |(f\circ g)(\gamma)_L\cap \gamma_R \cap f(\gamma)_R|
		\]
		and each summand can be rewritten as follows.
		\begin{align*}
			|(f\circ g)(\gamma)_L\cap \gamma_R \cap f(\gamma)_L| &= |\gamma_R \cap f(\gamma)_L| - |(f\circ g)(\gamma)_R\cap \gamma_R \cap f(\gamma)_L|\\
			|(f\circ g)(\gamma)_L\cap \gamma_R \cap f(\gamma)_R| &= |(f\circ g)(\gamma)_L \cap f(\gamma)_R| - |(f\circ g)(\gamma)_L\cap \gamma_L \cap f(\gamma)_R| \\
			&= |g(\gamma)_L \cap \gamma_R| - |(f\circ g)(\gamma)_L\cap \gamma_L \cap f(\gamma)_R| 
		\end{align*}
		Similarly, 
		\[
		|(f\circ g)(\gamma)_R\cap \gamma_L| = |(f\circ g)(\gamma)_R\cap \gamma_L \cap f(\gamma)_L| + |(f\circ g)(\gamma)_R\cap \gamma_L \cap f(\gamma)_R|
		\]
		and
		\begin{align*}
			|(f\circ g)(\gamma)_R\cap \gamma_L \cap f(\gamma)_L| &= |(f\circ g)(\gamma)_R \cap f(\gamma)_L| - |(f\circ g)(\gamma)_R\cap \gamma_R \cap f(\gamma)_L|\\
			&= |g(\gamma)_R \cap \gamma_L| - |(f\circ g)(\gamma)_R\cap \gamma_R \cap f(\gamma)_L| \\
			|(f\circ g)(\gamma)_R\cap \gamma_L \cap f(\gamma)_R| &= |\gamma_L \cap f(\gamma)_R| - |(f\circ g)(\gamma)_L\cap \gamma_L \cap f(\gamma)_R|
		\end{align*}
		Therefore, 
		\begin{align*}
			\Phi(f\circ g) &= |(f\circ g)(\gamma)_L\cap \gamma_R| - |(f\circ g)(\gamma)_R\cap \gamma_L|\\
			&= |\gamma_R \cap f(\gamma)_L| + |g(\gamma)_L \cap \gamma_R|  -  |g(\gamma)_R \cap \gamma_L| - |\gamma_L \cap f(\gamma)_R|\\
			&= \Phi(f) + \Phi(g)
		\end{align*}
		\item Suppose $f$ satisfies the condition.
		Then there exists $k$ so that $\eta\circ f \circ \eta^{-1}(\gamma) = \gamma$ where $\eta = \eta_{A, B, z}^k$.
		Thus, $\Phi(f) = \Phi(\eta\circ f \circ \eta^{-1}) = 0$.
	\end{enumerate}
\end{proof}
(1) and (3) in Proposition \ref{prop:properties_of_Phi} implies that $\eta_{A, B, z}$ is not generated by the maps which permutes $z$'s only finitely many.
Notice that the converse of (3) is false if one considers infinitely many half twists that the supports are mutually disjoint.
Therefore, by (3) $\Phi_z$ does not depend on the choice of types of shifts as we discussed in Section \ref{sec:shifts}.

\begin{rmk}
	By construction, $\Phi_z(\eta_{A, B, z'}) = 1$ if and only if $z \sim z'$.
	In other words, $\Phi_z$ detects how many shifts of $z$ occur between $A, B$.
	This detection also works for other types of shift maps.(recall Definition \ref{defn:shift_maps} that there are $3$ types of generalized shifts.)
	Proposition \ref{prop:not_full} implies that even if the shift is not full, it differs only by the surface homeomorphism $T$ to the full shift.
	Since $T$ is generated by half twists (or $1/n$ Dehn twists), $\Phi_z(T)$ is equal to $0$ by (3).
\end{rmk}

Using the same construction, we define $\Phi_z$ for each distinct $z \in E_{cp}(A, B)$.
Comprehensively, we have the following result.

\begin{thm}\label{thm:Theta_z}
	Let $S$ be an infinite type surface with countable end space.
	Suppose $\Map(S)$ is CB generated.
	Choose $z \in E_{cp}(A, B)$ for some maximal ends $A, B$.
	Let $\{A_0, \cdots, A_{N-1}\}$ be the set of maximal ends which admits $z$ as its immediate predecessor.
	We define $\Theta_z : \FMap(S) \to \mathbb{Z}^{N-1}$ as follows.
	\[
	\Theta_z : f \mapsto \big(\Phi_{z_i}(f)\big)_{1 \leq i \leq N-1}
	\]
	where $\Phi_{z_i} := \Phi_z$ for $z\in E_{cp}(A_0, A_i)$. 
	Then the following short exact sequence splits.
	\[
	1 \to \ker(\Theta_z) \to \FMap(S) \xrightarrow{\Theta_z} \mathbb{Z}^{N-1} \to 1
	\]
\end{thm}
\begin{proof}
	Consider the standard basis $\{e_i\}$ of $\mathbb{Z}^{N-1}$, where $z \in E_{cp}(A_0, A_{i})$ is an end corresponding to the $i$th entry.
	Let $\eta_i := \eta_{A_0, A_i, z}$ .
	Then $\Theta(\eta_i) = e_i$.
	Since all $\eta_i$'s commute each other, the sequence splits.
\end{proof}
Collecting all types of immediate predecessors, we have the following.
\begin{thm}\label{thm:Theta}
	Let $\Theta$ be the direct sum of $\Theta_z$'s for distinct types of immediate predecessors $z$.
	Then 
	\[
		1 \to \ker(\Theta) \to \FMap(S) \xrightarrow{\Theta} \bigoplus_{z} \mathbb{Z}^{N_z - 1} \to 1
	\]
	splits, where $N_z$ is the number of maximal ends which admits $z$ as an immediate predecessor.
\end{thm}
\begin{proof}
	It is enough to show that the shift map for distinct immediate predecessors $z$'s commutes with each other.
	This is true because we could realize the infinite strip $\mathcal{B}$ to be disjoint for distinct $z$'s.
\end{proof}
Remark that we could still extract a handle shift $h_{A, B}$ between two maximal ends $x_A, x_B$ if (1) $x_A, x_B$ are both ends accumulated by genus and (2) every $z \in E_{cp}(A, B)$ is not accumulated by genus.

As a result, we get the following.
\begin{thm}\label{thm:semi_directprod_of_FMap}
	Suppose $S$ is a tame, infinite type surface.
	Let $G_0 := \{ x \in \mathcal{M}(S) ~|~ x \in E^G \text{ and } Im(x) \cap E^G(S) = \emptyset\}$.
	Then, 
	\[
		\FMap(S) = \mathcal{F} \rtimes \left( \prod_{\substack{A, B \in \mathcal{M}(S) \\ \text{if } z \in E_{cp}(A, B)}} \langle \eta_{A, B, z} \rangle \oplus \prod_{\substack{A, B \in \mathcal{M}(S) \\ \text{if both } A, B \in G_0}} \langle h_{A, B} \rangle \right)
	\]
	where, the $\mathcal{F}$ is a subgroup which contains all finitely bounded mapping classes.
\end{thm}
The only term we add from Theorem \ref{thm:Theta} is the handle shift.
Handle shifts arise between any maximal ends accumulated by genus, but what it matters is the handle shift that any of $z \in Im(x)$ is not accumulated by genus.
We remark that $x \in G_0$ implies that the handle is in $Im(x)$.

We have now proved that if there are at least $2$ distinct maximal end types, we get the following.
\begin{prop}\label{prop:max_2_types_FMap}
	Suppose $S$ is a tame, infinite type surface with countable $E(S)$ and $\Map(S)$ is CB generated.
	If there are at least $2$ maximal end types in $E(S)$, then $\FMap(S)$ is not topologically normally generated.
\end{prop}
\begin{proof}
	Let $A, B$ be distinct maximal ends of $S$.
	If $\{A, B\} = \mathcal{M}(S)$, then either $W_A$ or $W_B$ is nonempty and this leads to a contradiction that $A$ does not satisfy the small zoom condition.
	Hence there must be another maximal end, denoted by $C$.
	It implies that there are at least $2$ elements $z, w$ in the disjoint union $E_{cp}(A, B) \cup E_{cp}(B, C) \cup E_{cp}(C, A)$ and thus we have a surjective homomorphism from $G$ to $\mathbb{Z}^2$ by restricting $\Phi$ only for $z, w$.
	By the lemma \ref{lem:vlamis_obs}, we are done.
\end{proof}

\section{From $\FMap(S)$ to $\Map(S)$}\label{sec:parity}
To disprove the normal generation of $\Map(S)$, the only thing left is to add the half twists.
The main problem dealing with the half twist is that it may permute the maximal end, so that we cannot construct the $\Theta$ map as before.
By Theorem \ref{thm:main_obstruction}, we restrict the surface $S$ which has only one type of maximal end that appears more than once.

Throughout this section, we will denote maximal ends of the same type $x$ to be $x_1, \cdots, x_n$ and the others to be $y_1 \cdots, y_m$.
Again by Theorem \ref{thm:main_obstruction}, all $y_i$'s are distinct from each other.

Since there are many ways to make a half twist, we begin with fixing the half twist for $x_i$'s.
By Proposition \ref{prop:prop5.4_mannrafi}, $S$ is decomposed as follows.
See Figure \ref{fig:contable_surface_decomp.png}.
\begin{itemize}
	\item Uniquely self-similar surfaces $K_1, \cdots, K_{n}$, and all $K_i$'s are homeomorphic to one another and its maximal end is of type $x_i$.
	\item Uniquely self-similar surfaces $K'_1, \cdots, K'_{m}$, and the maximal end of each $K'_j$ is of type $y_j$.
	\item A finite type surface $L$, which is homeomorphic to a $2$-sphere with $n+m$ boundaries.
\end{itemize}

\begin{figure}
	\centering
	\includegraphics[width=.5\textwidth]{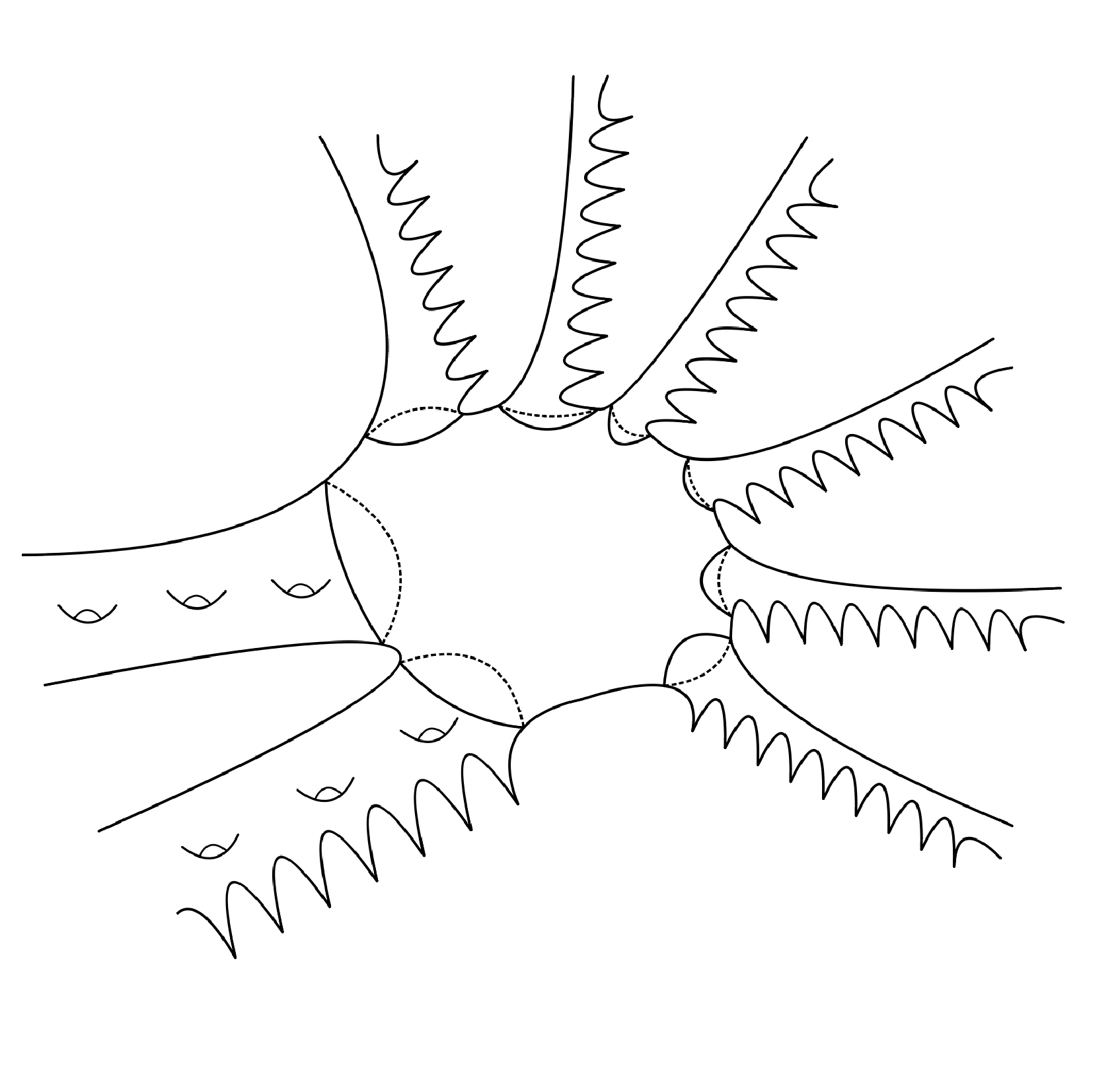}
	\caption{Surface decomposition for countable end space. Here we have $K_1, \cdots, K_6$ of unique maximal end $x \sim \omega + 1$, $K_1'$ of the end accumulated by genus only, and $K_2'$ of end accumulated by both genus and punctures.}
	\label{fig:contable_surface_decomp.png}
\end{figure}

Let $g_{i,j}$ be a half twist that exchanges $x_i$ and $x_j$.
We choose the representative of $g_{i,j}$ induced by the half twist in $L$. 

First observe that there are shift maps which are in the same conjugacy class.
\begin{lem}\label{lem:same_conjugacy_xy}
	For any $1\leq i, j \leq n$ and any $z \in E_{cp}(x_i, y_k)$, $\eta_{x_i, y_k, z}$ and $\eta_{x_j, y_k, z}$ are in the same conjugacy class if $z \in E_{cp}(x_i, y_k)$.
\end{lem}
\begin{proof}
	$\eta_{x_i, y_k, z} = g_{i,j} \circ \eta_{x_j, y_k, z} \circ g_{i,j}$.
\end{proof}
Similarly, we have the following.
\begin{lem}\label{lem:same_conjugacy_xx}
	Suppose $z$ is an immediate predecessor of $x$.
	For any $1 \leq i, j \leq n$, $\eta_{x_i, x_j, z} = g_{i,j} \circ \eta_{x_i, x_j, z}^{-1} \circ g_{i,j}$.
\end{lem}
\begin{lem}\label{lem:parity_map}
	For any $f \in \Map(S)$, there exists $g$, a finite product of half twists, so that $fg \in \FMap(S)$.
\end{lem}
\begin{proof}
	The image of $f$ under the forgetful map will give a permutation on the maximal ends.
	Since the permutation is realized as a finite product of half twists, we are done.
\end{proof}

We now ready to prove Theorem \ref{thm:thmalpha_countable}.

\begin{thm}\label{thm:Map_gen}
	Suppose $S$ is an infinite type surface and $\Map(S)$ is CB generated.
	Then $S$ is uniquely self-similar if and only if $\Map(S)$ is topologically normally generated.
\end{thm}
\begin{proof}
	We will use the notation as we discussed at the beginning of this Section \ref{sec:parity}.
	If $S$ is uniquely self-similar, then $S$ has a Rokhlin property.
	So we are done.
	
	For the converse, we split into $2$ cases.
	\begin{enumerate}
		\item (Case $1$) There are more than $1$ maximal type in $\mathcal{M}(S)$.

		It suffices to show that $n \geq 2$, otherwise $\Map(S) = \FMap(S)$ and it reduces to the Proposition \ref{prop:max_2_types_FMap}.
		Suppose there is $y \in  \{y_1, \cdots, y_m\}$ that $E_{cp}(x_i, y)$ is empty.
		Then there must be at least one $y' \in \{y_1, \cdots, y_m\} - \{y\}$ such that $E_{cp}(y, y_j)$ is nonempty, unless $S$ does not satisfy the small zoom condition.
		Construct the homomorphism from $\Map(S)$ to $(\mathbb{Z} / 2\mathbb{Z}) \times \mathbb{Z} $ as follows.
		\[
			^\forall f \in \Map(S),\quad f \mapsto (\mathcal{P}(f), \Phi_{y'}(f))
		\]
		$\mathcal{P}$ measures the parity of the permutation on $\{x_1, \cdots, x_n\}$, which sends odd permutations to $1$ and even permutations to $0$, induced by the forgetful map $\Map(S) \twoheadrightarrow S_n \twoheadrightarrow \mathbb{Z} / 2\mathbb{Z}$.
		
		Now suppose that every $y \in \{y_1, \cdots, y_m\}$ satisfies that $E_{cp}(x_i, y)$ is nonempty.
		Choose $y = y_j$ and let $z' \in E_{cp}(x_i, y)$.
		By Lemma \ref{lem:same_conjugacy_xy}, $\eta_{x_i, y, z'}$ and $\eta_{x_j, y, z'}$ must have the same image under any group homomorphism to an abelian group.
		Again, we construct the homomorphism from $\Map(S)$ to $(\mathbb{Z} / 2\mathbb{Z}) \times \mathbb{Z} $ as follows.
		\[
			^\forall f \in \Map(S),\quad f \mapsto (\mathcal{P}(f), \widehat{\Phi}_{z'}(f))
		\]
		$\mathcal{P}$ measures the parity as above, and $\widehat{\Phi}_{z'}$ measures the number by which $z'$ is either shifted from or shifted to the subsurface $L \cup (\cup_{1 \leq i \leq n} K_i)$ towards $K'_j$.
		In both cases, $\Map(S)$ is not normally generated by Lemma \ref{lem:vlamis_obs}, 
		
		\item (Case $2$) All maximal ends in $S$ are of the same type $x$.
		
		Let $i, j$ be distinct integers in $ \{1, \cdots, n\}$.
		Choose $z \in Im(x)$ and let $\eta$ be a shift map of $z$ from $x_i$ to $x_j$.
		If $Im(x)$ only contains a handle, then $x$ must be a type of ends which is accumulated only by genus and take the shift map $\eta$ as a handle shift from $x_i$ to $x_j$.
		Recall that $\Theta_z : \FMap(S) \to \mathbb{Z}^{n-1}$.
		Define $\widetilde{\Theta_z}: \FMap(S) \to \mathbb{Z}^{n-1} \to \mathbb{Z} \to \mathbb{Z} / 2\mathbb{Z}$,
		where the second map is the sum of all entries in $(n-1)$ tuple and the last map is a usual quotient map by $2\mathbb{Z}$.
		
		Now we construct the homomorphism.
		Let $f \in \Map(S)$ and define $\widetilde{\Theta}: \Map(S) \to (\mathbb{Z} / 2\mathbb{Z}) \times (\mathbb{Z} / 2\mathbb{Z})$ as follows.
		\[
			\widetilde{\Theta}(f) = \big( \widetilde{\Theta_z}(fg), sgn(g) \big)
		\]
		where $g \in \Map(L)$ such that $fg \in \FMap(S)$.
		This is a surjective group homomorphism, hence again by Lemma \ref{lem:vlamis_obs} we are done.
	\end{enumerate}
\end{proof}
We will prove that the map is in fact well-defined, surjective group homomorphism.
Since $\widetilde{\Theta}$ is surjective, $\Map(S)$ is not normally generated by Lemma \ref{lem:vlamis_obs}.
The remaining part of the proof is the well-definedness of $\widetilde{\Theta}$.
\begin{prop}\label{prop:parity_map}
	$\widetilde{\Theta}$ is well-defined surjective homomorphism.
\end{prop}
\begin{proof}
	Let $f \in \Map(S)$.
	We first show that $\widetilde{\Theta_z}$ is invariant under conjugation.
	Conjugation by an elements in $\FMap(S)$ does not change the value, by its definition.
	Hence, the only conjugation which may invoke any differences is by half twists.
	Conjugation by half twists differs value to either permute entries in the $(n-1)$ tuple, or change signs of some entries.
	Therefore, taking the sum and quotient by $2\mathbb{Z}$ is invariant.
	
	Suppose there are two $g_1, g_2$ such that both $fg_1, fg_2$ are in $\FMap(S)$. 
	It is obvious that $sgn(g_1) = sgn(g_2)$.
	Also, $fg_2 = fg_1 g_1^{-1}g_2$ and the support of $g_1^{-1}g_2$ lies in $L$.
	Hence the choice of half twists does not affect the value of $\widetilde{\Theta_z}$.
	
	By construction, $\widetilde{\Theta_z}$ is a group homomorphism.
	Now, consider $f_1, f_2 \in \Map(S)$.
	Then there are $g_1, g_2$ such that $f_ig_i \in \FMap(S)$ for $i = 1, 2$.
	Also, choose $g$ so that $f_1f_2g \in \FMap(S)$.
	It is obvious that $sgn(g) = sgn(g_1)+sgn(g_2) \mod 2\mathbb{Z}$.
	\begin{align*}
		\widetilde{\Theta_z}(f_1f_2g) &= \widetilde{\Theta_z}(f_1g_1) + \widetilde{\Theta_z}(g_1^{-1}f_2g_2g_1) + \widetilde{\Theta_z}(g_1^{-1}g_2^{-1}g)\\
		&= \widetilde{\Theta_z}(f_1g_1) + \widetilde{\Theta_z}(f_2g_2). 
	\end{align*}
	The term $\widetilde{\Theta_z}(g_1^{-1}g_2^{-1}g)$ vanishes since it is supported on finitely bounded subsurface.
	Therefore $\widetilde{\Theta_z}$ is a group homomorphism.
	It is easy to check that $\widetilde{\Theta}$ is surjective.	
\end{proof}

\begin{rmk}
	We remark that just capturing the parity of the shift and the sign of permutation on the maximal end does not work in general.
	For example, Let $S$ is a Jacob ladder surface and let $\sigma$ be an involution that is a $\pi$-rotation fixing one separating curve, and $\tau$ be another involution that is again a $\pi$-rotation fixing a separating curve pair, as in Figure \ref{fig:parity_remark}.
	If we merely measure how many genus are shifted and the sign of the permutation on the ends, then both $\sigma$ and $\tau$ are mapped to $(0,1)$.
	But $\sigma \tau$ is a handle shift, hence $\sigma \tau$ is mapped to $(1,0)$ and it fails to be a group homomorphism.
	However, if we designate the half twist in advance (in this case, the only half twist $g$ is $\sigma$), then $\widetilde{\Theta}(\tau) = (1,1)$.
\end{rmk}

\begin{figure}[h]
	\centering
	\includegraphics[width=.7\textwidth]{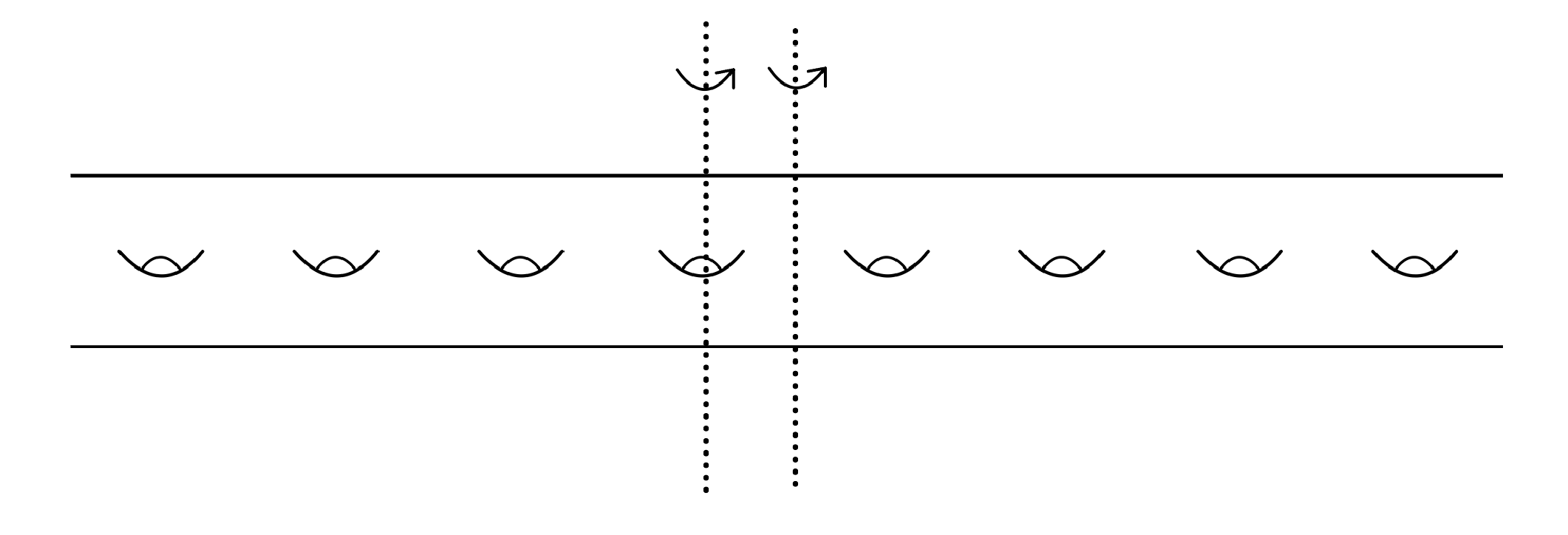}
	\caption{Each $\sigma$(right axis) and $\tau$(left axis) is the $\pi$ rotation along the axis in the middle, respectively.}
	\label{fig:parity_remark}
\end{figure}

\section{Uncountable end space}\label{sec:uncountable}

In this section we provide an example of surfaces with \emph{uncountable end space} which admit topologically normally generated mapping class group.

The goal of the section is to prove the following.
\begin{thm}\label{thm:uncountable_case_ex}
	Suppose $S$ is a tame, infinite type surface and satisfies the following.
	\begin{enumerate}
		\item $E(S)$ is uncountable and $\Map(S)$ is CB generated.
		\item $S$ is a connected sum of $S_{p}$ and $S_{u}$, where $S_p$ is perfectly self-similar and $S_u$ is uniquely self-similar with unique maximal end $x$, and $|Im(x)| \leq 1$. 
	\end{enumerate}
	Then $\Map(S)$ is topologically normally generated.
\end{thm}
Note that since we can choose any kind of perfectly self-similar tame surface, there are uncountably many examples satisfying the conditions in Theorem \ref{thm:uncountable_case_ex}.

\begin{figure}[h]
	\centering
	\includegraphics[width=.4\textwidth]{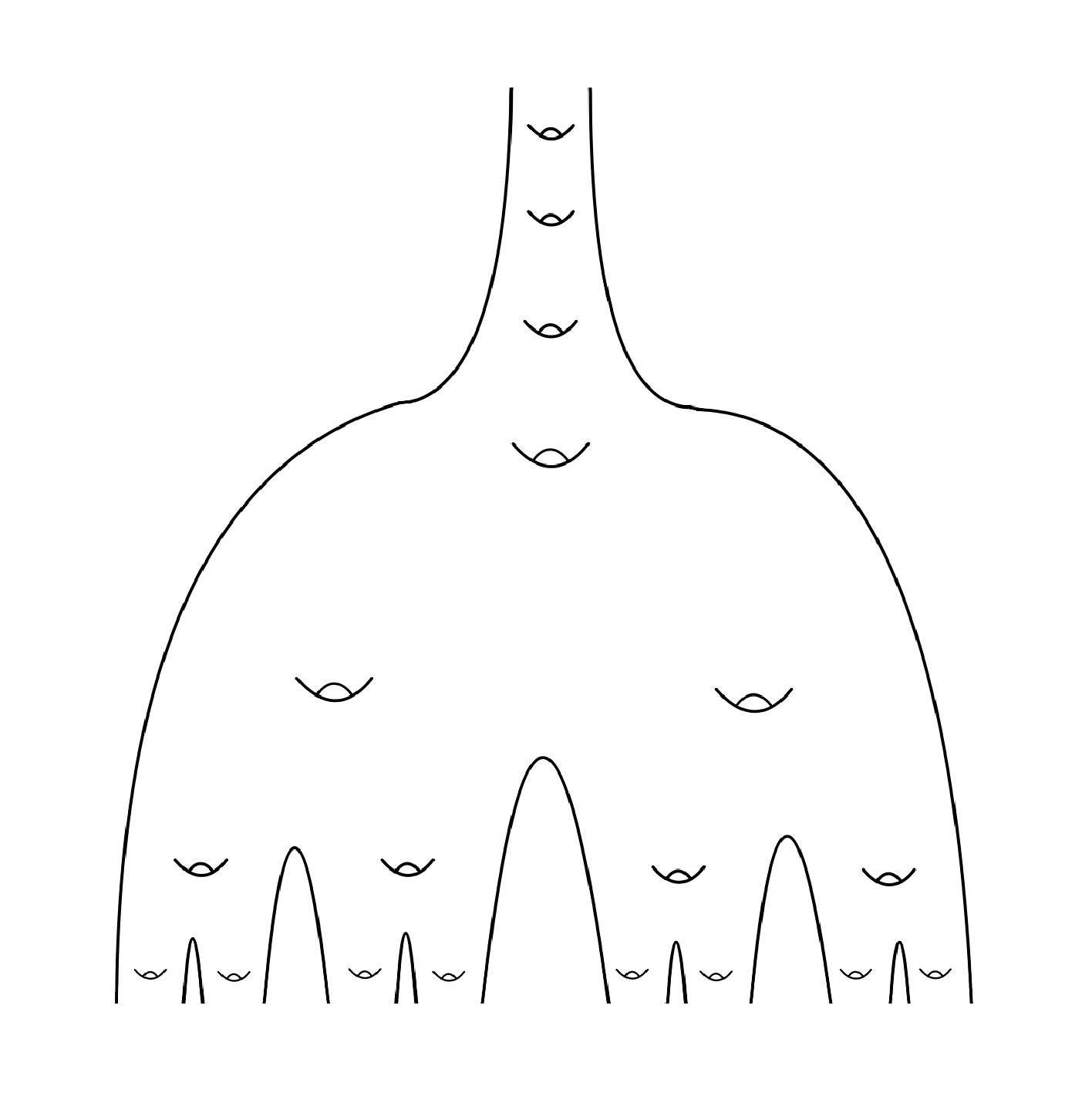}
	\caption{Example of surfaces that satisfies the condition in Theorem \ref{thm:uncountable_case_ex}.}
	\label{fig:blooming_cantor_with_one_iso}
\end{figure}

\subsection{Half-spaces}
We first introduce the half-space. 
In \cite{malestein2024selfsimilar}, the half-space was discussed only in the perfectly self-similar surface.
However, their argument works well on the surface if it is a connected sum with a perfectly self-similar subsurface.
For more details about half-spaces, we refer to \cite{malestein2024selfsimilar}. 

Let $\Sigma$ be perfectly self-similar subsurface with one boundary in $S$.
Then $H\subset \Sigma$ is called a \emph{half-space} of $\Sigma$ if it satisfies the following.
\begin{enumerate}
	\item $H$ is closed subset of $\Sigma$ and $\partial H$ is connected and compact.
	\item $E(H)$ and $E(\overline{H}^c\cap \Sigma)$ both contain a maximal end of $E(\Sigma)$.
\end{enumerate}
By the nature of perfectly self-similar surface, it is possible to find countably many homeomorphic copies of $H$ in $\Sigma$ that are pairwise disjoint to one another.
\begin{lem}\label{lem:H-translation}
	Let $\Sigma$ be perfectly self-similar and $\{H_i\}_{i \in \mathbb{Z}}$ be a sequence of pairwise disjoint half-spaces.
	Then there is $\varphi \in \Homeo(\Sigma)$ such that $\varphi(H_i) = H_{i+1}$, $i \in \mathbb{Z}$.
\end{lem}
In other words, a $H$-translation map is a shift map that shifts the copies of $H$.
Such $\varphi$ is called a \emph{$H$-translation map}.

\begin{lem}{\cite[Lemma 3.10]{malestein2024selfsimilar}}\label{lem:half-space_displacable}
	Let $H, H'$ be half-spaces of $\Sigma$.
	Then there is a homeomorphism $\mu$ of $\Sigma$ so that $\mu(H) = H'$.
\end{lem}

\begin{lem}{\cite[Lemma 3.7]{malestein2024selfsimilar}}\label{lem:lemma3.7_malestein_tao}
	Suppose $H$ is a half-space and $\varphi$ is a $H$-translation map.
	The normal closure of $\varphi$ contains $\Homeo(H, \partial H)$.
\end{lem}
\begin{proof}
	Let $H := H_0$.
	Choose any $f \in \Map(H, \partial H)$.
	Consider the repetition map $\widetilde{f} = f \circ (\varphi^{-1} f \varphi) \circ  (\varphi^{-2} f \varphi^2) \circ \cdots$.
	Since each support of $\varphi^{-k} f \varphi^k$ is in $H_k$, hence $\widetilde{f}$ is well-defined.
	Then we have $f = \widetilde{f} \varphi^{-1} \widetilde{f}^{-1} \varphi$, which implies that $f$ is in the normal closure of $\varphi$. 
	Since $f$ is arbitrary, $\Homeo(H, \partial H) \subset \langle\langle \varphi \rangle\rangle$.
\end{proof}
The technique we introduced in the proof of Lemma \ref{lem:lemma3.7_malestein_tao} is referred to as a \textit{swindle}.
With Lemma \ref{lem:half-space_displacable}, the half space in $H$ can be arbitrary.

\subsection{Proof of Theorem \ref{thm:uncountable_case_ex}}
Let $S$ be a surface which satisfies the assumption in Theorem \ref{thm:uncountable_case_ex}.
Denote $x$ a unique isolated maximal end, and let $z$ be the only immediate predecessor of $x$ if exists.

We first decompose the surface $S$ as follows.
Choose a subsurface $K$ embedded in $S$ which is homeomorphic to a sphere with $4$ boundaries, such that $S - K$ has $4$ connected components consisting of $S_u$ deleted with an open disk and $3$ homeomorphic copies of $S_p$ deleted with an open disk.
By abusing the notation, we label each subsurface with $S_u$ and $S_1, S_2, S_3$, respectively.
Note that the boundary of $S_u$ separates the $S_u$ part from the $S_p$ part.
See Figure \ref{fig:topnormgen_uncountable_surf_decomp} also.

\begin{figure}[h]
	\centering
	\includegraphics[width=.4\textwidth]{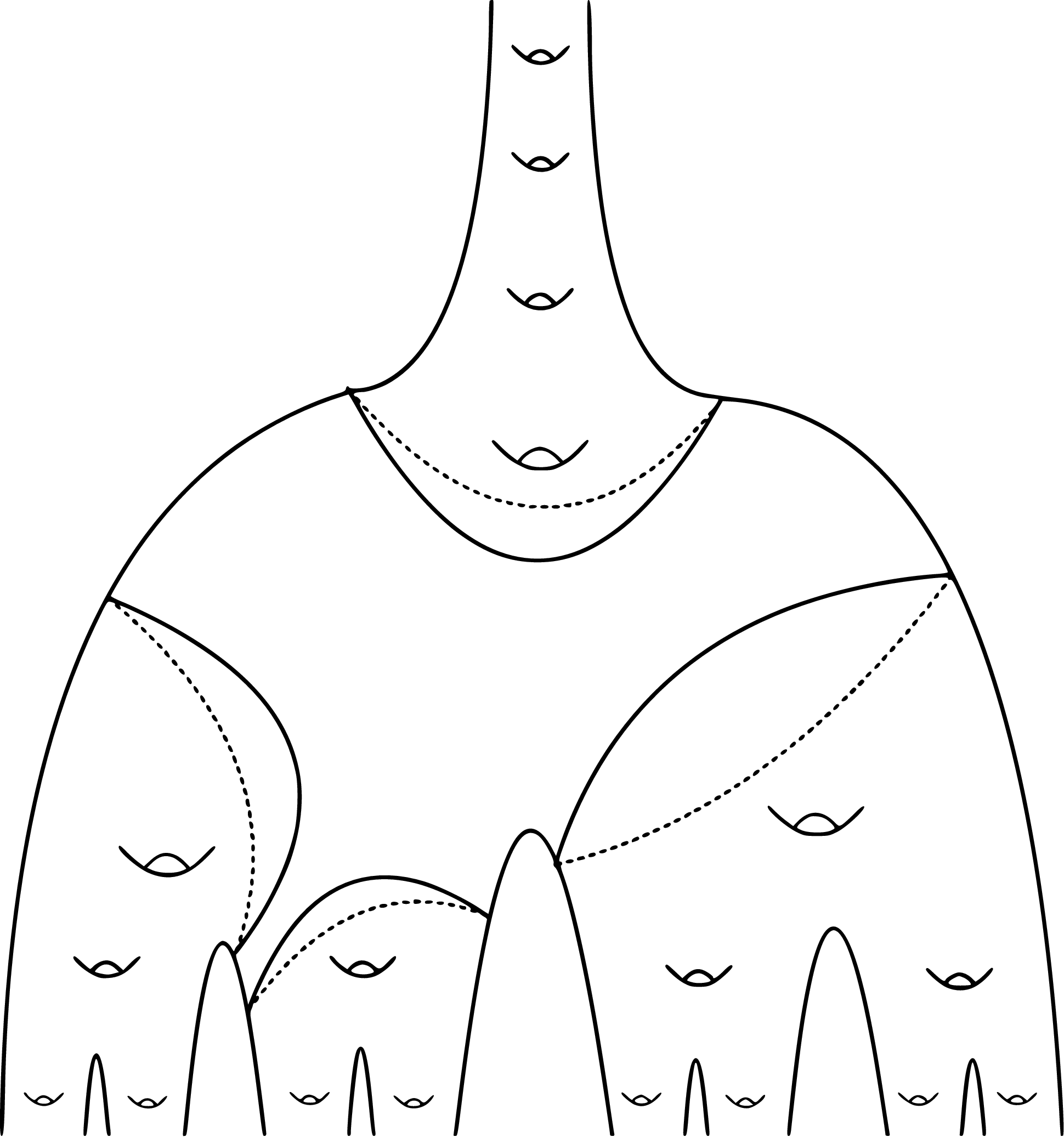}
	\caption{The surface decomposition by the $2$-sphere with $4$ boundaries at the center. There is a $S_u$ on the top, and the other $3$ subsurfaces are $S_1, S_2, S_3$, all homeomorphic to each other.}
	\label{fig:topnormgen_uncountable_surf_decomp}
\end{figure}

\begin{lem}\label{lem:uncountable_section,half_twist}
	Let $H$ be a half-space in $S_p$ and $\varphi$ be a $H$-translation map.
	The half twist between $S_i, S_j$, $i\neq j \in \{1, 2,3\}$ is in the normal closure of $\varphi$.
\end{lem}
\begin{proof}
	Without loss of generality, we choose $S_1, S_2$. 
	Note that the union of $S_1, S_2$ is contained in the half space $H'$ which is disjoint with $S_3$.
	The half twist $g_{12}$ is supported in $H'$.
	By Lemma \ref{lem:half-space_displacable}, choose a homeomorphism $\mu$ of $S_p$ such that $\mu(H) = H'$.
	Then $\mu^{-1} g_{12}\mu \in \Homeo(H, \partial H)$.
	By Lemma \ref{lem:lemma3.7_malestein_tao}, $g_{12} \in \langle\langle \varphi \rangle\rangle$.  
\end{proof}
Remark that the half-space $H$ and $\varphi$ were arbitrary in Lemma \ref{lem:uncountable_section,half_twist}.
Therefore, we can choose a sufficiently small half-space $H_0$ which contains both $H$ and the support of $\varphi$.
Let $x$ be the unique maximal end of $S_u$.
Also, we choose a maximal end $y$ so that $y \not\in E(H_0)$.
Suppose $z \in E_{cp}(x, y)$ exists and denote $\eta_z$ to be a generalized shift map from $x$ to $y$.
Under conjugation by an appropriate half twist, any generalized shift map conjugates to $\eta_z$.

\begin{lem}\label{lem:uncountable_section,etavar}
	$\langle\langle \eta_z\varphi \rangle\rangle$ contains both $\eta_z$ and $\varphi$.
\end{lem}
\begin{proof}
	By the choice of $y$ and $\varphi$, $\eta_z$ and $\varphi$ commutes.
	As in the \textit{swindle} proof, any element in $\Map(H, \partial H)$ is a commutator.
	Hence, if the half-space $H$ does not intersect with the support of $\eta_z$, then $\Map(H, \partial H)$ is also contained in $\langle\langle \eta_z\varphi \rangle\rangle$. 
	Thus, $\varphi \in \langle\langle \eta_z\varphi \rangle\rangle$. 
\end{proof}

Lemma \ref{lem:uncountable_section,etavar} implies that every mapping class group whose support is contained in the union of $S_1, S_2, S_3$ is in $\langle\langle \eta_z\varphi \rangle\rangle$.

\begin{lem}\label{lem:uncountable_section,Map(K)}
	$\langle\langle \eta_z\varphi \rangle\rangle$ contains $\Map(K)$.
\end{lem}
\begin{proof}
	Note that $\Map(K)$ is generated by the Dehn twists along the boundaries and separating curves.
	Each Dehn twist along the boundary of $S_i$ is a mapping class in $\Map(S_i, \partial S_i)$, by pushing the support slightly beyond $K$.
	Also, the Dehn twist along the boundary of $S_u$ can be made by the \textit{swindle} technique with $\eta_z$, by copying the Dehn twists towards to $x$. 
	The remaining part is a Dehn twist along seperating curves.
	Since this is the square of half twists $g_{ij}$, we are done.
\end{proof}

$S_u$ is uniquely self-similar with a boundary, hence it is not uniquely self-similar.
Choose any homeomorphism $f \in \Map(S_u, \partial S_u)$.
After conjugating by $\eta_z$, the support $\eta_z^{-1} f \eta_z$ is proper subset of $S_u$.
Thus, any $f$ can be considered as a mapping class of $\hat{S_u}$, where $\hat{S_u}$ is a disk capping surface of $S_u$ along $\partial S_u$.
Since $\hat{S_u}$ is uniquely self-similar, there is a dense conjugacy class.

Before we continue the argument, we first introduce which element has a dense conjugacy class.
We follow the notation in \cite{lanier2023homeomorphism}.
Let $S$ be uniquely self-similar surface and $x$ be a maximal end of $S$.
\begin{prop}[Proposition 4.2 in \cite{lanier2023homeomorphism}]	
	Let $c$ be a separating curve of $S$ and $\Omega_c$ denote the component of $S - c$ such that $x \in E(\Omega_s)$.
	Suppose $G$ be a subgroup of $\Map(S)$ as follows.
	\[
		G := \left\{ g \in \Map(S) ~:~ g\big\vert_{\Omega_c} = id \textrm{ for some separating curve } c ~\right\}.
	\]
	Then $G$ is dense in $\Map(S)$.
\end{prop}

Let $\mathbb{D}$ be a closed disk embedded in $S$.
For every separating curve $c$ in $S$, there is an isotopy class which avoids $\mathbb{D}$.
Hence $G$ is a collection of finitely bounded mapping classes and moreover it is dense in $\Map(S - int(\mathbb{D}))$. 
\begin{lem}\label{lem:uncountable_Map(Su)}
	Let $G$ be as above and consider it as a subgroup of $\Map(S_u, \partial S_u)$.
	Then $G$ is contained in $\langle\langle \eta_z \rangle\rangle$.
\end{lem}
\begin{proof}
	By Proposition \ref{prop:swindle}, every finitely bounded map is a commutator of $\eta_z^k$ and a repetition map.
\end{proof}

Comprehensively, we prove Theorem \ref{thm:uncountable_case_ex}.
\begin{proof}[Proof of Theorem \ref{thm:uncountable_case_ex}]
	Suppose $S$ is a surface that satisfies the assumption.
	Decompose the surface as we described in the beginning of the subsection. (See Figure \ref{fig:topnormgen_uncountable_surf_decomp}.)
	The decomposition satisfies Proposition \ref{prop:prop5.4_mannrafi}, hence $\Map(S)$ is generated by the identity neighborhood $\mathcal{V}_K$, $\Map(K)$, $\eta_z$ (if $z$ exists), handle shifts (if they exist) and the half twists among $S_1, S_2, S_3$.
	By Lemma \ref{lem:uncountable_section,etavar}, $\langle\langle \eta_z\varphi \rangle\rangle$ contains both $\eta_z$ and $\varphi$.
	So by Lemma \ref{lem:lemma3.7_malestein_tao}, \ref{lem:uncountable_Map(Su)}, $\mathcal{V}_K \subset \langle\langle \eta_z\varphi \rangle\rangle$.
	Also, by Lemma \ref{lem:uncountable_section,half_twist}, \ref{lem:uncountable_section,Map(K)}, both $\Map(K)$ and half twists are contained in $\langle\langle \eta_z\varphi \rangle\rangle$.
	Therefore, $\Map(S)$ is topologically normally generated by $\eta_z\varphi$.
\end{proof}

The remaining case is that $S$ is a perfectly self-similar with one puncture.
This is already answered in \cite{malestein2024selfsimilar}.
\begin{thm}[{\cite[Theorem B]{malestein2024selfsimilar}}]
	Let $S$ be a perfectly self-similar surface with a puncture.
	Then $\Map(S)$ is normally generated by a single involution.
\end{thm}

\begin{rmk}
	We remark that the choice of $\eta_z$ and $\varphi$ only depends on the choice of $y$, which is an end of the perfectly self-similar part.
	Thus all such $\eta_z\varphi$ are all in the same conjugacy class.
	\begin{que}
		Suppose $S$ satisfies the assumption of Theorem \ref{thm:uncountable_case_ex}.
		Is there a normal generator that is not conjugate to $\eta_z \varphi$?
	\end{que}
	
	Secondly, there are many infinite type surfaces with uncountable end space, whose mapping class group is not topologically normally generated.
	In principle, the main obstruction of being topologically normally generated is the existence of $2$ distinct shifts from an isolated maximal end to a maximal end in a perfectly self-similar subsurface.
	If they exist, then by measuring the flux we could construct a surjective homomorphism from $\Map(S)$ to $\mathbb{Z}^2$.
	However, if the isolated maximal ends are all punctures, then we cannot detect the normal generation since there are no flux.
	\begin{que}\label{que:S-F_top_normed}
		Let $S$ be a perfectly self-similar and $F$ be a finite subset in $S$ with $|F| \geq 2$.
		Is $\Map(S - F)$ topologically normally generated?  
	\end{que}
\end{rmk}

\section{Normal generating sets}\label{sec:n_of_norm_gen}
In the previous sections, we had a continuous surjection of $\Map(S)$ onto an abelian group $G$.
If $n$ be a minimum number of generators of $G$, then $\Map(S)$ needs at least $n$ number of elements to normally generates $\Map(S)$, since each generator in $G$ has a preimage of a certain conjugacy class.

Let $M_{iso}$ be a number of distinct types of isolated maximal ends in $S$ that is not a puncture.
\begin{cor}\label{cor:lower_bound_of_normal_gens}
	$\Map(S)$ is normally generated by at least $\max(1, M_{iso}-1)$ elements.
\end{cor}
\begin{proof}
	Define $\mathcal{M}_2(S)$ to be the set of isolated maximal types such that each end has at least $2$ different ends of such type in $S$.
	Recall that there is a forgetful map.
	\[
		\Map(S) \twoheadrightarrow S_{n_1} \times \cdots \times S_{n_k} \twoheadrightarrow \underbrace{(\mathbb{Z} / 2\mathbb{Z}) \times \cdots \times (\mathbb{Z} / 2\mathbb{Z})}_{M_2}
	\]
	and the number $(\mathbb{Z} / 2\mathbb{Z})$ on the rightmost term is equal to $M_2 := |\mathcal{M}_2(S)|$.
	
	There are $M_1 := M_{iso} - M_2$ number of isolated maximal ends which are unique of this type in $S$.
	Denote such ends as $x_1, \cdots, x_{M_1}$.
	By the small zoom condition, there must exist $z_1, \cdots, z_{M_1}$, where each $z_i$ is an immediate predecessor of $x_i$ and $\eta_{z_i}$ is a shift map pushing $z_i$ towards to $x_i$.
	Since all $x_i$'s are distinct from each other, we can choose at least $M_1 -1$ number of $\eta_{z_i}$'s that any two of them are not in the same conjugacy class.
	Consider the $\Phi$ map as constructed in Section \ref{sec:cohomology}, which only measures the $\eta_{z_i}$ flux.
	
	Combine the forgetful map and $\Phi$.
	We have 
	\[
		\Map(S) \twoheadrightarrow \bigoplus_{M_1-1} \mathbb{Z} \times \bigoplus_{M_2} (\mathbb{Z} / 2\mathbb{Z})
	\]
	Therefore, we need at least $M_{iso}-1$ elements to generate $\Map(S)$.   
\end{proof}

To generate $\Map(S)$ we first collect all generalized shifts first.
Define $C$ as below.
\[
	C := \max_{A, B \in \mathcal{M}(S)} |E_{cp}(A, B)|
\]
Since $S$ is tame and $\Map(S)$ is CB generated, $C$ is a finite integer which only depends on the topology of $S$.
\begin{prop}\label{prop:upper_bound_of_shifts}
	The number of generalized shift is bounded by $CM$.
\end{prop}
\begin{proof}
	For each type of maximal end $x$ of $S$, there are at most $C$ number of generalized shifts (up to conjugating the maximal ends of the same type) which send an immediate predecessor towards to $x$.
	Thus, the number of generalized shift is bounded by $CM$.
\end{proof}
Let $\mathcal{S}$ to be the collection of the following mapping classes.
\begin{enumerate}
	\item Suppose $x, y \in \mathcal{M}(S)$ and both $x, y$ are isolated.
	Then for any $z \in E_{cp}(x, y)$, $\eta_z \in \mathcal{S}$.
	\item Suppose $x, y \in \mathcal{M}(S)$ and $E(x)$ is a Cantor set.
	Choose one H-translation map $\varphi$ for $x$ whose closure of ends of the support is not 
	Then for any $z \in E_{cp}(x, y)$, $\eta_z \varphi \in \mathcal{S}$.
\end{enumerate}
By Proposition \ref{prop:upper_bound_of_shifts}, $|\mathcal{S}| \leq CM$.

Let $\mathcal{M}$ be the set of maximal ends of $S$.
For every $x \in \mathcal{M}$, $E(x)$ is either finite or a Cantor set.
Hence $\mathcal{M}$ is partitioned into two set, $\mathcal{M}_u := \{x \in \mathcal{M} ~|~ |E(x)| < \infty\}$ and $\mathcal{M}_u := \{x \in \mathcal{M} ~|~ E(x) \text{ is a Cantor set}\}$.
We decompose $S$ as follows.
\begin{enumerate}
	\item Choose a closed subsurface $L$ embedded in $S$ such that $S - \text{int}(L)$ is a disjoint union of finitely many subsurfaces $\{A_i\}$ satisfying that
	\begin{itemize}
		\item The maximal ends of each $A_i$ are of the same type.
		\item If $x \in \mathcal{M}$ and $x\in \mathcal{M}(A_i)$, then for any $y \sim x$, $y \in \mathcal{M}(A_i)$.
	\end{itemize}
	\item Let $K$ be a closed subsurface such that $S - \text{int}(K)$ is a disjoint union of finitely many subsurfaces $\{B_i\}$, satisfying that 
	\begin{itemize}
		\item $L \subset K$.
		\item Every connected component of $S - \text{int}(K)$ is a self-similar surface with open disk deleted, and the union of the maximal ends of the components is equal to $\mathcal{M}(S)$. 
		\item  If $x \in \mathcal{M}_p$ and $x \in \mathcal{M}(B_j)$, then $B_j$ is a perfectly self-similar surface with open disk deleted, and there are exactly $2$ more subsurfaces $B_l, B_m$ which are both homeomorphic to $B_j$.
	\end{itemize}
\end{enumerate}
Remark that the surface decomposition by $K$ satisfies the condition in Proposition \ref{prop:prop5.4_mannrafi}.

\begin{figure}[h]
	\centering
	\includegraphics[width=.5\textwidth]{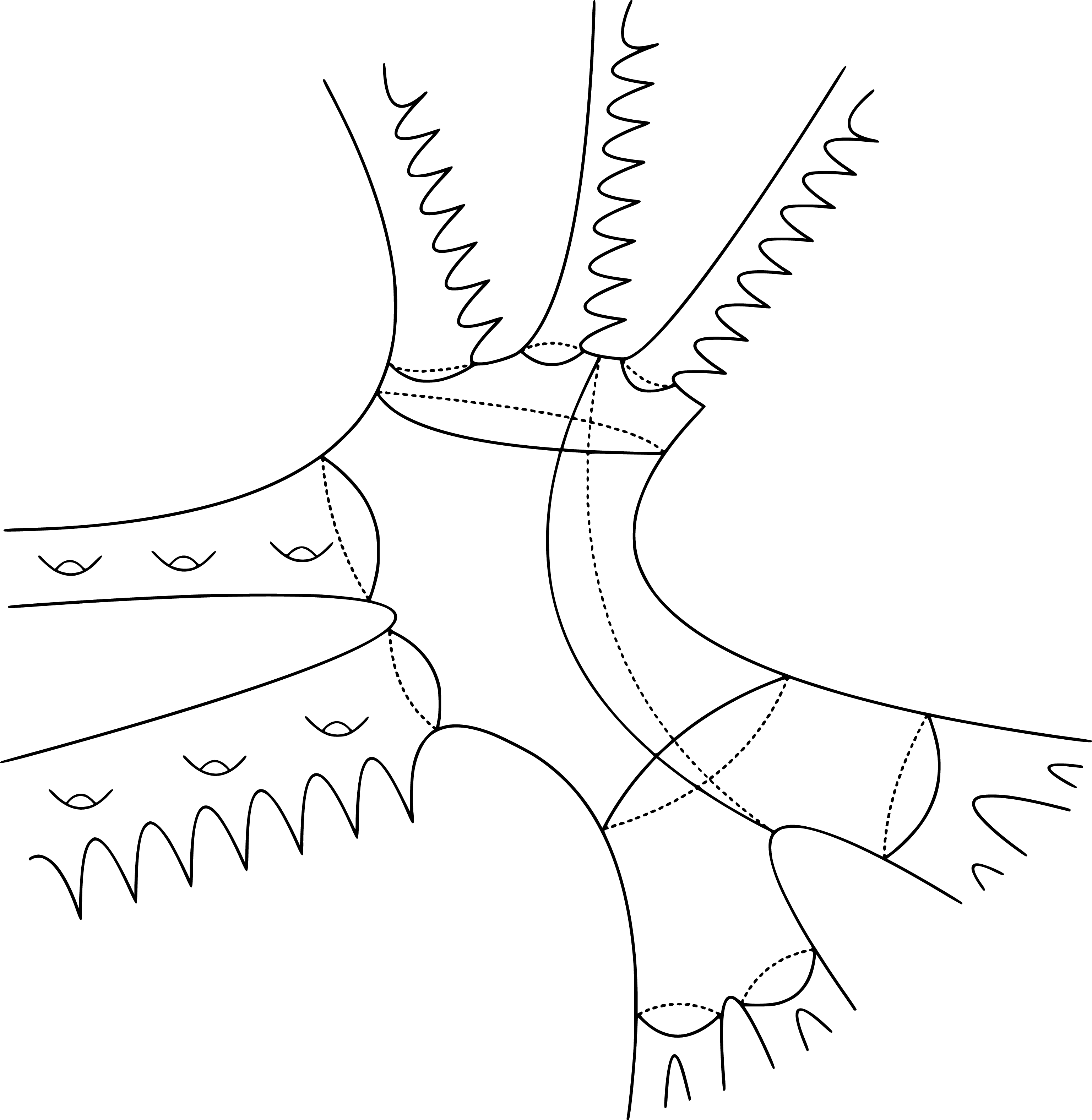}
	\caption{The surface decomposition. The curve in the middle which passes two boundaries of $L$ is an example of $\gamma$ in the proof of Proposition \ref{prop:Map(K)}.}
	\label{fig:uncountable_surface_decomp}	
\end{figure}

\begin{prop}\label{prop:dehn_along_boundaries}
	Let $K$ be the subsurface as above.
	Every Dehn twist along the boundary of $K$ is in $\overline{\langle\langle \mathcal{S} \rangle\rangle}$.
\end{prop}
\begin{proof}
	Let $T_\gamma$ be a Dehn twist along the curve $\gamma$.
	If the boundary $\gamma$ bounds perfectly self-similar subsurface, then by Lemma \ref{lem:uncountable_section,etavar} $\overline{\langle\langle \mathcal{S} \rangle\rangle}$ contains the Dehn twist along the boundary, since it is supported in some half space.
	
	Let $\gamma$ be the boundary which bounds the isolated maximal end, say $x$.
	We will use the \emph{swindle}.
	Let $\eta_x$ to be a product of $\eta_z$, which tends to $x$ for all $z \in Im(x)$.
	Consider the separating curves $\{\gamma_i\}_{i \geq 1}$ such that $\gamma_1$ is the boundary of $K$ and 
	$\gamma_{i+1} = \eta_x(\gamma_i)$.
	Let $\widetilde{T_{\gamma}}:= \prod_{i = 1}^\infty T_{\gamma_i}$.
	Then $T_{\gamma_1} = \eta_x \widetilde{T_{\gamma}}^{-1} \eta_x^{-1} \widetilde{T_{\gamma}}$.
	Hence $T_{\gamma_1}$ is in the normal closure of $\eta_x$, which is a product of shifts in $\mathcal{S}$.
\end{proof}

\begin{prop}\label{prop:Map(K)}
	Let $G$ be a subgroup of $\Map(K, \partial K)$ whose elements induce the mapping class in $\Map(S)$, except the Dehn twists along the boundary of $K$.
	Then $G$ is normally generated by $M(M-2)$ number of elements.
\end{prop}
\begin{proof}
	Note that $K$ is a sphere deleted with finitely many pairwise disjoint open disks. 
	Hence the subgroup $G$ is generated by Dehn twists along the boundaries and the half and full Dehn twists with a few exceptions.
	Such exceptions are the mapping classes related to the half twists along the curve intersecting the boundary of $L$, so that the maximal ends beyond the curve are not of the same type.
		
	Suppose $\gamma$ is a separating curve that intersects $2$ boundaries of $L$ and one of the connected components of $K - \gamma$ is a pair of pants.
	See also Figure \ref{fig:uncountable_surface_decomp}.
	The conjugacy class of $T_\gamma$ depends only on which boundaries $\gamma$ meets.
	This is because, if we have another $\gamma'$, then there are appropriate half twists $g, g'$ supported in $K-\text{int}(L)$ so that $gg'(\gamma) = \gamma'$.
	Therefore we only need $\frac{M(M-1)}{2}$ curves for each pair of connected components of $K-\text{int}(L)$ .
	Also if we choose another curve intersecting $2$ boundaries of $L$, then the Dehn twist along the curve can be generated by the lantern relation.
	Again by the lantern relation, the Dehn twist along any boundary of $L$ is generated by the half twists which are supported on the component of $K - \text{int}(L)$ which contains the boundary.
	
	The remaining part is $T_\gamma$ for separating curve $\gamma$ contained in $L$.
	Since $L$ is a sphere deleted with $M$ number of pairwse disjoint open disks, we only need $\frac{M(M-3)}{2}$ number of Dehn twists.
	Therefore, we have $\frac{M(M-1)}{2} + \frac{M(M-3)}{2} = M(M-2)$ number of generators to generate $G$.
\end{proof}
Let $\mathcal{D}$ denote the set of Dehn twists we choose in Proposition \ref{prop:Map(K)}.

For handle shifts, there are many ways to construct a handle shift which are not conjugate to each other.
For example, consider one maximal end $x_1$ that is accumulated by both punctures and genus, and another maximal end $x_2$ is accumulated by the ends that is accumulated by genus.

Let $h_1$ be a handle shift which shifts handles from $x_1$ to $x_2$, and $h_2$ be a handle shift which shifts handles from $x_1$ to the immediate predecessors of $x_2$.
Since the target end of $h_1$ and $h_2$ are not the same, two cannot be in the same conjugacy class.

Still, every handle shift is normally generated with finitely many handle shifts.
\begin{prop}\label{prop:upper_bound_of_handle_shifts}
	Any handle shift can be normally generated by $M$ number of handle shifts.
\end{prop}
To prove Proposition \ref{prop:upper_bound_of_handle_shifts}, we need additional lemmas.

Suppose $x$ is an isolated maximal end of an infinite type surface $S$ with CB generated $\Map(S)$.
Under the surface decomposition the uniquely self-similar subsurface $\Sigma_x$ for $x$ is a countable connected sum of a surface $W_{x}$, where $W_x$ is a connected sum of $2$-sphere with $2$ boundaries and $|Im(x)|$ number of self-similar surfaces that each maximal end of the subsurface is one of immediate predecessors of $x$. \footnote{If the handle is in $Im(x)$, then we add one genus in $W_x$.}

We refer to $W^n_x$ as $n \geq \mathbb{Z}_{\geq 0}$ and to boundaries as $\partial_+^n$, $\partial_-^n$.
We fix the orientation reversing homeomorphism $\lambda : \partial_+^n \to \partial_-^n$ once and for all.
Then $\Sigma_x$ can be made by attaching $\partial_+^n$ to $\partial_-^{n+1}$ along $\lambda$ for all $n \geq \mathbb{Z}_{\geq 0}$.
Let $c_n$ be an image of $\partial_+^n$ in $\Sigma_x$.

\begin{defn}\label{defn:finite_bound}
	$f \in \FMap(\Sigma_x)$ is \emph{finitely bounded} if there are integers $k_1, k_2 > 0$ and a representative of $f$ whose support is bounded by a subsurface cobounded by $c_{k_1}$ and $c_{k_2}$.
	Suppose $S$ is an infinite type surface and $\Map(S)$ is CB generated.
	Let $x$ be one of the maximal end of $S$. 
	$f \in \FMap(S)$ is \emph{finitely bounded at $x$} if $f$ is induced by a finitely bounded map under inclusion of $\Sigma_x \xhookrightarrow{} S$.
\end{defn}
For example, any generalized shift is not finitely bounded.
Also, if $f$ is finitely bounded, then for any $z$, $\Phi_z(f) = 0$.

\begin{lem}\label{lem:shift_makes_finite_bound_map}
	Suppose $f$ is finitely bounded on $x$.
	Then $f$ is normally generated by generalized shifts and handle shifts.
\end{lem}
\begin{proof}
	We will use the \emph{swindle} technique again.
	Since $f$ is finitely bounded on $x$, there are $2$ end separating curves labeled with $k_1 < k_2$.
	Choose the generalized shifts $\eta_z$s for all immediate predecessors $z$ of $x$.
	Without loss of generality, we assume that each $\eta_z$ shift $z$ towards to $x$, for all $z$.
	Let $\eta := \prod_{z} \eta_z$ and $k = k_2 - k_1$. 
	We construct $F$ as follows.
	\[
		F := f \circ (\eta^{k} f  \eta^{-k}) \circ (\eta^{2k} f  \eta^{-2k}) \circ \cdots .
	\]
	Note that the support of $(\eta^{ik}f \eta^{-ik})$ overlaps to $(\eta^{jk}f \eta^{-jk})$ if and only if $i = j$, which implies that they commute each other.
	\[
		f = F \circ (\eta^{k} \circ F \circ \eta^{-k})^{-1}
	\]
	Hence, $f$ is normally generated by $\eta$.
\end{proof}

\begin{proof}[proof of Proposition \ref{prop:upper_bound_of_handle_shifts}]
	Let $h$ be a handle shift of $S$ and let $\mathfrak{g}$ be a handle lying in the support of $h$.
	We define $i(h)$ to be the end accumulated by $h^{-n}, n\geq 1$ image of $\mathfrak{g}$.
	Similarly, we define $t(h)$ to be the end accumulated by $h^{n}, n \geq 1$ image of $\mathfrak{g}$.
	There are $4$ types of handle shifts in $S$. 
	\begin{itemize}
		\item[Case $(1)$] $i(h_1)$ is maximal in $E(S)$ and $t(h_1)$ is an immediate predecessor of some maximal end of $S$.
		
		\item[Case $(2)$] $i(h_2)$ is maximal and $t(h_2)$ neither maximal in $E(S)$ nor an immediate predecessor of any maximal ends in $S$.
		Such $h_2$ is a composition of $h_1$ and a handle shift $h'$ whose $i(h') = t(h_1)$ and $t(h') = t(h_2)$.
		Since such $h'$ is finitely bounded, $h_2$ is normally generated by generalized shifts and $h_1$.
		
		\item[Case $(3)$] Both $i(h_3)$ and $t(h_3)$ are neither maximal in $E(S)$ nor immediate predecessors of any maximal ends in $S$.
		Such $h_3$ is a composition of two different handle shifts of Case $(2)$.
		
		\item[Case $(4)$] Both $i(h_4), t(h_4)$ are maximal in $E(S)$.
		Such $h_4$ is a composition of two different handle shifts of Case $(1)$.
	\end{itemize}
	Hence, what we need is a collection of handle shifts of the Case $(1)$ for each maximal end. 
	Therefore we need at most $M$ handle shifts.
\end{proof}
Let $\mathcal{H}$ denote the handle shifts we choose in Proposition \ref{prop:upper_bound_of_handle_shifts}.

The last part is $\mathcal{V}_{K}$.
First we note that $\mathcal{V}_K$ is a product of $\Map(B_i, \partial B_i)$, where $\{B_i\}$ is a collection of connected components of $S - \text{int}(K)$.
For $B_i$, whose $\mathcal{M}(B_i)$ is a Cantor set, $\Map(B_i, \partial B_i)$ is already contained in the $\topnorm{\mathcal{S}}$, since $\mathcal{S}$ has an H-translation map.
Also, if $B_j$ contains a unique maximal end, $\Map(B_j, \partial B_j)$ is also contained in the $\topnorm{\mathcal{S}}$, since $\mathcal{S}$ is generated by generalized shifts.
Therefore we have the following.
\begin{thm}\label{thm:n_of_generator}
	Suppose $S$ is a tame, infinite type surface with CB generated $\Map(S)$.
	Let $M$ be the number of distinct types of maximal ends in $S$, \emph{i.e.,}
	\[
		M = \big\vert \mathcal{M}(S)/\sim \big\vert
	\]
	Also, let $C := \max_{A, B \in \mathcal{M}(S)} |E_{cp}(A, B)|$.
	Then $\Map(S) = \topnorm{\mathcal{S}\cup \mathcal{D} \cup \mathcal{H}}$.
	\emph{i.e.,} $\Map(S)$ is topologically normally generated by $M(M+C-1)$.
\end{thm}

Combining with Corollary \ref{cor:lower_bound_of_normal_gens}, we have the following upper and lower bound on $\mathbf{n}(S)$.
\[
	\max(1,M_{iso} - 1) \leq \mathbf{n}(S) \leq M(M+C-1)
\]
We remark that $\frac{\mathbf{n}(S)}{(M+C)^2}$ is bounded. 
\emph{i.e.,} the minimal number of topological normal generators grows quadratically by the sum of $2$ constants that only depend on the topology of the surface.

We also remark that both $C$ and $M$ could be arbitrary, by its definition.
Also, because we take the lower bound depending on the number of types of isolated maximal ends, $\frac{\mathbf{n}(S)}{(M+C)^2}$ might be arbitrarily small.
For example, if $S$ has infinitely many different types of maximal ends $x_i$ such that $E(x_i)$'s are all Cantor sets, and have $1$ isolated maximal end, then $k = 1$.
However, since $M$ is arbitrary, $\frac{\mathbf{n}(S)}{(M+C)^2}$ could be arbitrarily small.

Bounds on the minimal number of topological normal generators gives a bound of rank of its abelianization.
By Field-Patel-Rasmussen \cite{field2022sclbigmcg}, they proved that the commutator subgroup of the mapping class group of surface $S - F$ we questioned above is closed is finitely generated and discrete with the quotient topology.
\begin{thm}{(Theorem 1.1, 1.4 in \cite{field2022sclbigmcg})}\label{thm:fpr_abelianization}
	Let $S$ be an infinite type surface with tame end space such that $\Map(S)$ is CB generated.
	Suppose $S$ is an infinite type surface that every equivalence class of maximal ends of $S$ is infinite, except (possibly) for finitely many isolated planar ends.
	Then,
	\begin{enumerate}
		\item The commutator subgroup $[\Map(S), \Map(S)]$ is Polish and both open and closed subgroup of $\Map(S)$.
		\item The abelianization of $\Map(S)$ is finitely generated and is discrete when endowed with the quotient topology.
	\end{enumerate}
\end{thm}
Applying their result of \cite{field2022sclbigmcg}, Theorem \ref{thm:n_of_generator} gives an upper bound of the rank.
\begin{cor}\label{cor:Z2rank_of_ab}
	Suppose $S$ is a finite connected sum of perfectly self-similar surfaces, possibly with finitely many punctures.
	Then the abelianization of $\Map(S)$ is finitely generated by at most $M(M+C-1)$.
\end{cor}
The proof is simple, since the number of topological normal generators can not exceed the number of generators of its abelianization.
However, we cannot get an effective lower bound by using the inequality in Theorem \ref{thm:thmalpha_lowerbound}, since $M_{iso}$ counts the isolated maximal ends which is not the same type of a puncture.

\section{Remarks on locally CB}\label{sec:locCB_not_CBgend}
In \cite{lanier2023vRP}, the authors proved that for an infinite type surface $S$, $\Map(S)$ has the virtual Rokhlin property (\emph{i.e.,} $\Map(S)$ contains a finite index subgroup which has the Rokhlin property) if and only if $S$ is spacious.
More precisely, they found an open subgroup $\Gamma(S)$ which is an open, finite-index subgroup of $S$ with the Rokhlin property.
We first recall the some definitions in \cite{lanier2023vRP}.
\begin{defn}\label{defn:spacious}~
	
	\begin{enumerate}
		\item A $2$-manifold $S$ is \emph{weakly self-similar} if every proper compact subset of $S$ is displaceable.
		\emph{i.e.,} such that for any compact subsurface $\Sigma \subset S$, there is a homeomorphism $\varphi$ of $S$ so that $\Sigma \cap \varphi(\Sigma) = \phi$.
		
		\item Suppose $S$ is a weakly self-similar and it is homeomorphic to $N$ or $N\# N$ for some uniquely self-similar surface $N$.
		Then $\mu \in E(S)$ is \emph{maximally stable} if there is a nested subsequence $\Sigma_n,n \in \mathbb{Z}_{\geq 0}$ such that $\Sigma_n$ is homeomorphic to $N$ with the open disk removed, $\Sigma_n$ is a neighborhood of $\mu$, and $\cap_{n \in \mathbb{Z}_{\geq 0}}$ is empty.
		Let $\Gamma(S)$ be a subgroup of $\Homeo_+(S)$ consisting of homeomorphisms whose support is outside of 
		neighborhoods of stable maximal ends.
		Note that $\Gamma(S)$ is an open normal subgroup of $\Homeo_+(S)$.
		
		\item A $2$-manifold $S$ is \emph{spacious} if it is $S$ is weakly self-similar and for any simple closed curve $c$ in $S$, $\Gamma(S) \backslash \Homeo_+(S) \cdot c$ is finite.
	\end{enumerate}
\end{defn} 

Together with the work of Mann-Rafi \cite{mann2023large}, they also give an uncountable family of spacious surfaces, which is the surface with countable ends that $E(S) = \omega^\alpha \cdot 2 + 1, ~E^G(S) = \phi$ for any countable limit ordinal $\alpha$.
Since $\alpha$ is a limit ordinal, all surfaces in the example are locally CB but not CB generated.
The simplest example is when $\alpha = \omega$, which is called a \emph{The Great Wave off Kanagawa surface}.
See Figure \ref{fig:wave_off_kanagawa}, a cartoon of the surface.

\begin{figure}[h]
	\centering
	\includegraphics[width=.5\textwidth]{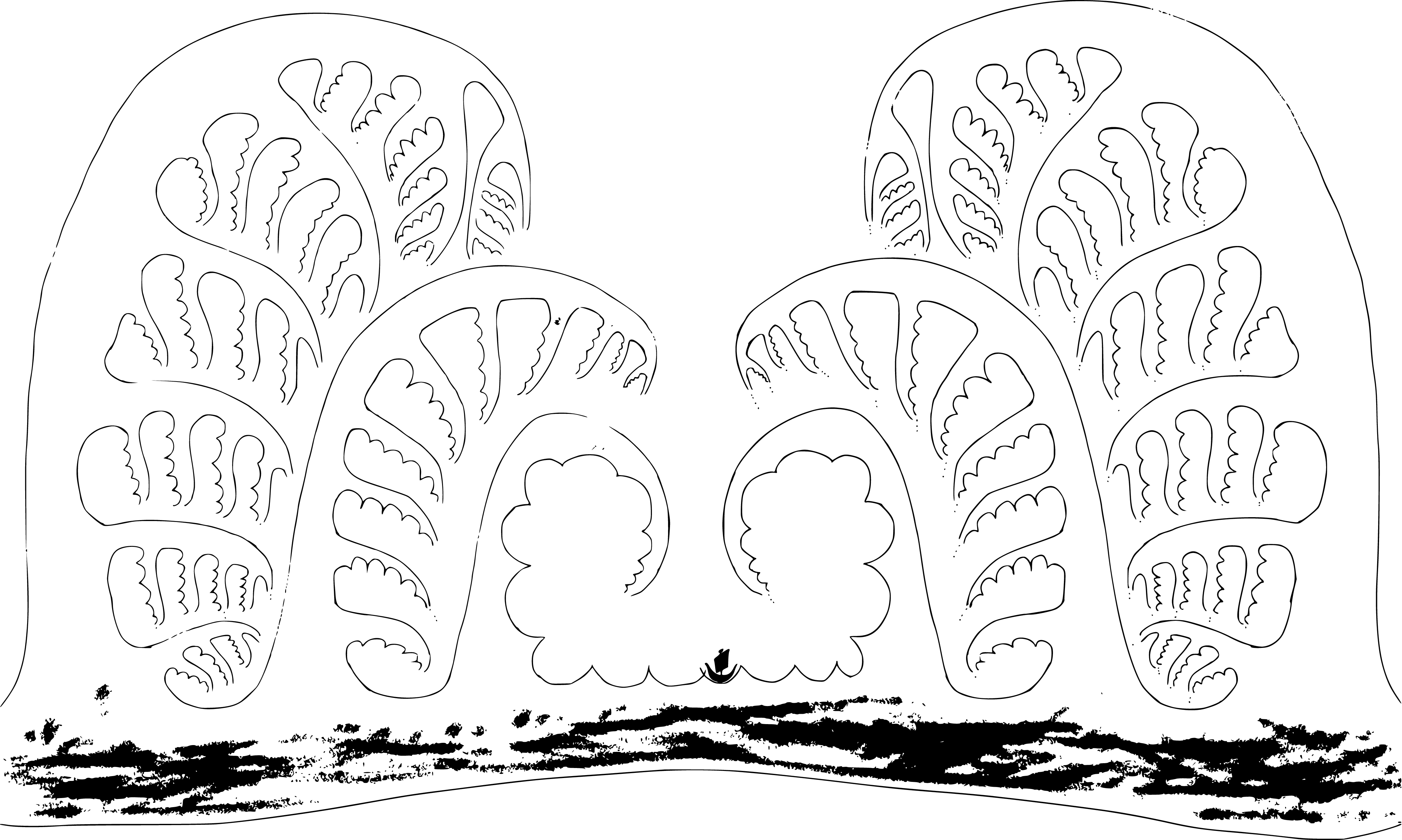}
	\caption{Cartoon of the Great Wave off Kanagawa surface $S$. $E(S) = \omega^\omega\cdot 2 + 1$ and $E^G(S) = \emptyset$.}
	\label{fig:wave_off_kanagawa}
\end{figure}

Also, the closure of $\Gamma(S)$ is $\FMap(S)$.
Hence it is normally generated by the element chosen from the dense conjugacy class of $\Gamma(S)$ and the maximal end flip.
\begin{prop}\label{prop:uncountable_example_of_loc_CB}
	There are uncountably many surfaces $S$ which are locally CB but not CB generated, and $\FMap(S)$ is topologically normally generated.
\end{prop}

We end the paper with questions about relaxing the tameness or CB generation.
\begin{que}
	Find a surface $S$ that $\Map(S)$ is locally CB but not CB generated, and topologically normally generated.
\end{que}
\begin{que}
	Find a non-tame surface $S$ that $\Map(S)$ is CB generated, and topologically normally generated.
\end{que}

In \cite{rolland2024large}, the authors proved that there exists an infinite type surface that is not tame, but admits a CB generated mapping class group (but not globally).
If the surface is not tame, the shift map may not be constructed in general.
As in the example, if we allow the surface $S$ to be of limit type or non-tame, (so that $\Map(S)$ is no longer CB generated) then we cannot construct the generalized shift map, which makes our method in the paper is not applicable.

\bibliographystyle{alpha} 
\bibliography{refs}

@article{patelvlamis2018algebraic,
	author = {Priyam Patel and Nicholas G Vlamis},
	title = {{Algebraic and topological properties of big mapping class groups}},
	volume = {18},
	journal = {Algebraic \& Geometric Topology},
	number = {7},
	publisher = {MSP},
	pages = {4109 -- 4142},
	year = {2018},
	doi = {10.2140/agt.2018.18.4109},
	URL = {https://doi.org/10.2140/agt.2018.18.4109}
}

@misc{rolland2024large,
	archiveprefix = {arXiv},
	author = {Rita Jim{\'e}nez Rolland and Israel Morales},
	eprint = {2309.05820},
	primaryclass = {math.GT},
	title = {On the large scale geometry of big mapping class groups of surfaces with a unique maximal end},
	url = {https://arxiv.org/abs/2309.05820},
	year = {2024}}

@misc{hernandez2021integral,
	archiveprefix = {arXiv},
	author = {Jes{\'u}s Hern{\'a}ndez Hern{\'a}ndez and Cristhian E. Hidber},
	eprint = {2104.02684},
	primaryclass = {math.GT},
	title = {First integral cohomology group of the pure mapping class group of a non-orientable surface of infinite type},
	url = {https://arxiv.org/abs/2104.02684},
	year = {2021}}

@book{rosendal2022coarse,
	author = {Rosendal, Christian},
	isbn = {978-1-108-84247-1},
	mrclass = {22-02 (22A05 54H11)},
	mrnumber = {4327092},
	pages = {ix+297},
	publisher = {Cambridge University Press, Cambridge},
	series = {Cambridge Tracts in Mathematics},
	title = {Coarse geometry of topological groups},
	url = {https://mathscinet.ams.org/mathscinet-getitem?mr=4327092},
	volume = {223},
	year = {2022}}

@article{malestein2024selfsimilar,
	author = {Malestein, Justin and Tao, Jing},
	doi = {10.1307/mmj/20216114},
	fjournal = {Michigan Mathematical Journal},
	issn = {0026-2285},
	journal = {Michigan Math. J.},
	mrclass = {20 (57)},
	mrnumber = {4767502},
	number = {3},
	pages = {485--508},
	title = {Self-{S}imilar {S}urfaces: {I}nvolutions and {P}erfection},
	url = {https://doi.org/10.1307/mmj/20216114},
	volume = {74},
	year = {2024}}

@article{field2022sclbigmcg,
	author = {Field, Elizabeth and Patel, Priyam and Rasmussen, Alexander J.},
	doi = {10.1112/blms.12707},
	fjournal = {Bulletin of the London Mathematical Society},
	issn = {0024-6093,1469-2120},
	journal = {Bull. Lond. Math. Soc.},
	mrclass = {57K20 (20F65 57M07)},
	mrnumber = {4561192},
	mrreviewer = {Hongbin\ Sun},
	number = {6},
	pages = {2492--2512},
	title = {Stable commutator length on big mapping class groups},
	url = {https://doi.org/10.1112/blms.12707},
	volume = {54},
	year = {2022}}

@incollection{aramayona2020bigmcgoverview,
	author = {Aramayona, Javier and Vlamis, Nicholas G.},
	booktitle = {In the tradition of {T}hurston---geometry and topology},
	doi = {10.1007/978-3-030-55928-1\_12},
	isbn = {978-3-030-55928-1; 978-3-030-55927-4},
	mrclass = {57K20},
	mrnumber = {4264585},
	pages = {459--496},
	publisher = {Springer, Cham},
	title = {Big mapping class groups: an overview},
	url = {https://doi.org/10.1007/978-3-030-55928-1_12},
	year = {[2020] \copyright2020}}

@article{vlamis2024telescoping,
	author = {Nicholas G. Vlamis},
	eprint = {2403.03887},
	month = {03},
	title = {Homeomorphism groups of telescoping 2-manifolds are strongly distorted},
	url = {https://arxiv.org/pdf/2403.03887.pdf},
	year = {2024}}

@article{lanier2023vRP,
	author = {Justin Lanier and Nicholas G. Vlamis},
	eprint = {2308.13138},
	month = {08},
	title = {Homeomorphism groups of 2-manifolds with the virtual Rokhlin property},
	url = {https://arxiv.org/pdf/2308.13138.pdf},
	year = {2023}}

@article{abbott2021infinite-type,
	author = {Carolyn R. Abbott and Nicholas Miller and Priyam Patel},
	eprint = {2109.06106},
	month = {09},
	title = {Infinite-type loxodromic isometries of the relative arc graph},
	url = {https://arxiv.org/pdf/2109.06106.pdf},
	year = {2021}}

@article{grant2021asymptotic,
	author = {Curtis Grant and Kasra Rafi and Yvon Verberne},
	eprint = {2110.03087},
	month = {10},
	title = {Asymptotic Dimension of Big Mapping Class Groups},
	url = {https://arxiv.org/pdf/2110.03087.pdf},
	year = {2021}}

@article{lanier2022normal,
	author = {Lanier, Justin and Margalit, Dan},
	doi = {10.4171/cmh/526},
	fjournal = {Commentarii Mathematici Helvetici. A Journal of the Swiss Mathematical Society},
	issn = {0010-2571,1420-8946},
	journal = {Comment. Math. Helv.},
	mrclass = {20F36 (20F38 57K20)},
	mrnumber = {4410724},
	mrreviewer = {Stefan\ Witzel},
	number = {1},
	pages = {1--59},
	title = {Normal generators for mapping class groups are abundant},
	url = {https://doi.org/10.4171/cmh/526},
	volume = {97},
	year = {2022}}

@article{lanier2022Rokhlin,
	author = {Lanier, Justin and Vlamis, Nicholas G.},
	doi = {10.1007/s00209-022-03096-3},
	fjournal = {Mathematische Zeitschrift},
	issn = {0025-5874,1432-1823},
	journal = {Math. Z.},
	mrclass = {57K20 (20E45 20F38 37D40 54H11)},
	mrnumber = {4492497},
	mrreviewer = {Mahender\ Singh},
	number = {3},
	pages = {1343--1366},
	title = {Mapping class groups with the {R}okhlin property},
	url = {https://doi.org/10.1007/s00209-022-03096-3},
	volume = {302},
	year = {2022}}

@article{aramayona2020firstcohomology,
	author = {Aramayona, Javier and Patel, Priyam and Vlamis, Nicholas G.},
	doi = {10.1093/imrn/rnaa229},
	fjournal = {International Mathematics Research Notices. IMRN},
	issn = {1073-7928,1687-0247},
	journal = {Int. Math. Res. Not. IMRN},
	mrclass = {57K20 (20F65)},
	mrnumber = {4216709},
	mrreviewer = {Xiao\ Ming\ Du},
	number = {22},
	pages = {8973--8996},
	title = {The first integral cohomology of pure mapping class groups},
	url = {https://doi.org/10.1093/imrn/rnaa229},
	year = {2020}}

@article{hernandez2022conjugacy,
	author = {Hern\'{a}ndez Hern\'{a}ndez, Jes\'{u}s and Hru\v{s}\'{a}k, Michael and Morales, Israel and Randecker, Anja and Sedano, Manuel and Valdez, Ferr\'{a}n},
	doi = {10.1112/jlms.12594},
	fjournal = {Journal of the London Mathematical Society. Second Series},
	issn = {0024-6107,1469-7750},
	journal = {J. Lond. Math. Soc. (2)},
	mrclass = {57K20 (03C30 20E45)},
	mrnumber = {4477213},
	mrreviewer = {Ser-Wei\ Fu},
	number = {2},
	pages = {1131--1169},
	title = {Conjugacy classes of big mapping class groups},
	url = {https://doi.org/10.1112/jlms.12594},
	volume = {106},
	year = {2022}}

@article{mann2023large,
	author = {Mann, Kathryn and Rafi, Kasra},
	doi = {10.2140/gt.2023.27.2237},
	fjournal = {Geometry \& Topology},
	issn = {1465-3060,1364-0380},
	journal = {Geom. Topol.},
	mrclass = {57K20 (57M07)},
	mrnumber = {4634747},
	number = {6},
	pages = {2237--2296},
	title = {Large-scale geometry of big mapping class groups},
	url = {https://doi.org/10.2140/gt.2023.27.2237},
	volume = {27},
	year = {2023}}

@article{domat2022pure,
	author = {Domat, George},
	fjournal = {Mathematical Research Letters},
	issn = {1073-2780,1945-001X},
	journal = {Math. Res. Lett.},
	mrclass = {57K20 (20F38 37E30)},
	mrnumber = {4516036},
	note = {Appendix with Ryan Dickmann},
	number = {3},
	pages = {691--726},
	title = {Big pure mapping class groups are never perfect},
	volume = {29},
	year = {2022}}

@article{calegari2022normal,
	author = {Calegari, Danny and Chen, Lvzhou},
	doi = {10.1090/btran/108},
	fjournal = {Transactions of the American Mathematical Society. Series B},
	issn = {2330-0000},
	journal = {Trans. Amer. Math. Soc. Ser. B},
	mrclass = {57K20 (20E07 20F05 20J06)},
	mrnumber = {4498366},
	mrreviewer = {Mustafa\ Korkmaz},
	pages = {957--976},
	title = {Normal subgroups of big mapping class groups},
	url = {https://doi.org/10.1090/btran/108},
	volume = {9},
	year = {2022}}

@article{richards1963classification,
	author = {Richards, Ian},
	doi = {10.2307/1993768},
	fjournal = {Transactions of the American Mathematical Society},
	issn = {0002-9947,1088-6850},
	journal = {Trans. Amer. Math. Soc.},
	mrclass = {54.75},
	mrnumber = {143186},
	mrreviewer = {S.\ S.\ Cairns},
	pages = {259--269},
	title = {On the classification of noncompact surfaces},
	url = {https://doi.org/10.2307/1993768},
	volume = {106},
	year = {1963}}

@book{kerekjarto1923vorlesungen,
	author = {Ker{\'e}kj{\'a}rt{\'o}, B{\'e}la},
	publisher = {J. Springer},
	title = {Vorlesungen {\"u}ber Topologie},
	volume = {8},
	year = {1923}}

@article{baik2021reducible,
	author = {Baik, Hyungryul and Kim, Dongryul M and Wu, Chenxi},
	journal = {arXiv preprint arXiv:2112.13726},
	title = {Reducible normal generators for mapping class groups are abundant},
	year = {2021}}

@article{lanier2023homeomorphism,
	author = {Lanier, Justin and Vlamis, Nicholas G},
	journal = {arXiv preprint arXiv:2308.13138},
	title = {Homeomorphism groups of 2-manifolds with the virtual Rokhlin property},
	year = {2023}}

@article{vlamis2023homeomorphism,
	author = {Vlamis, Nicholas G},
	journal = {arXiv preprint arXiv:2306.08619},
	title = {Homeomorphism groups of self-similar 2-manifolds},
	year = {2023}}
	
\end{document}